\documentclass[hidelinks,onefignum,onetabnum]{siamart250211}

\usepackage{lipsum}
\usepackage{amsfonts}
\usepackage{graphicx}
\usepackage{epstopdf}
\usepackage{algorithmic}
\ifpdf
  \DeclareGraphicsExtensions{.eps,.pdf,.png,.jpg}
\else
  \DeclareGraphicsExtensions{.eps}
\fi

\newsiamremark{remark}{Remark}
\newsiamremark{hypothesis}{Hypothesis}
\crefname{hypothesis}{Hypothesis}{Hypotheses}
\newsiamthm{claim}{Claim}
\newsiamremark{fact}{Fact}
\crefname{fact}{Fact}{Facts}

\headers{Neural surrogates for RTE}{
}

\title{Uniformly accurate structure-preserving neural surrogates for radiative transfer
\thanks{Submitted to the editors DATE: \today.
\funding{Jingrun Chen is supported by NSFC Major Research Plan -  Interpretable and General-purpose Next-generation Artificial Intelligence (No 92370205), NSFC grant 12425113, and the Key Laboratory of the Ministry of Education for Mathematical Foundations and Applications of Digital Technology, University of Science and Technology of China. 
Keke Wu is supported by the China Postdoctoral Science Foundation under Grant Number 2025M773105 and the Jiangsu Funding Program for Excellent Postdoctoral Talent.
}}}

\author{
Mengjia Bai\thanks{School of Mathematical Sciences, Fudan University, Shanghai, China
(\email{bai\_mj@fudan.edu.cn}).}
\and Jingrun Chen\thanks{School of Mathematical Sciences and Suzhou Institute for Advanced Research, University of Science and Technology of China, and Suzhou Big Data \& AI Research and Engineering Center, China (\email{jingrunchen@ustc.edu.cn}).}
\and Keke Wu\thanks{School of Artificial Intelligence and Data Science and Suzhou Institute for Advanced Research, University of Science and Technology of China, and Suzhou Big Data \& AI Research and Engineering Center, China (\email{wukekever@ustc.edu.cn}).}
}

\usepackage{amsopn}

\usepackage{amssymb,amsmath,amssymb}
\usepackage{mathrsfs}
\usepackage{graphicx}
\usepackage{subfigure}
\usepackage{caption}
\usepackage{float}
\usepackage{epstopdf}
\usepackage{makecell, rotating, multirow,diagbox}
\usepackage{booktabs}
\usepackage{makecell}
\usepackage{color}
\usepackage{url}
\usepackage[shortlabels]{enumitem}

\def\bx{\boldsymbol{x}}

\newsiamremark{assumption}{Assumption} 
\crefname{assumption}{Assumption}{Assumption} 

\begin{document}
\maketitle
\begin{abstract}
In this work, we propose a uniformly accurate, structure-preserving neural surrogate for the radiative transfer equation with periodic boundary conditions based on a multiscale parity decomposition framework. The formulation introduces a refined decomposition of the particle distribution into macroscopic, odd, and higher-order even components, leading to an asymptotic-preserving neural network system that remains stable and accurate across all parameter regimes. By constructing key higher-order correction functions, we establish rigorous uniform error estimates with respect to the scale parameter $\varepsilon$, which ensures $\varepsilon$-independent accuracy. Furthermore, the neural architecture is designed to preserve intrinsic physical structures such as parity symmetry, conservation, and positivity through dedicated architectural constraints. The framework extends naturally from one to two dimensions and provides a theoretical foundation for uniformly accurate neural solvers of multiscale kinetic equations. Numerical experiments confirm the effectiveness of our approach.

\end{abstract}



\section{Introduction}
Radiative transfer describes the propagation and interactions of radiation or particles within participating media~\cite{modest2021radiative,larsen1987asymptotic}. 
This field of study has broad applications across disciplines such as medical imaging~\cite{suetens2017fundamentals,klose2002optical1,klose2002optical2}, cancer therapy~\cite{baskar2012cancer,ren2006frequency,arridge2009optical}, and astrophysical flows~\cite{mihalas2013foundations}.
It is also essential for analyzing energy transfer in engineering systems -- for example, neutron transport in fission reactors~\cite{davison1958neutron} or photon-mediated heating in inertial confinement fusion capsules~\cite{lewis1984computational}.

The numerical solution of the radiative transfer equation (RTE), however, poses several well-known challenges.
First, the solution is defined over a high-dimensional phase space, typically encompassing three spatial coordinates and three momentum (or angular) variables, and it may also exhibit time dependence. 
Second, the interaction between radiation and matter can be highly complex, especially when the material properties themselves evolve. 
Third, spatial and temporal heterogeneities in collision rates, arising from strongly energy-dependent processes or material discontinuities, introduce multiple characteristic scales. 
To address these difficulties, various numerical strategies have been developed, broadly categorized into deterministic and stochastic methods. Among deterministic approaches, the discrete ordinates method (DOM/DVM), also known as the $S_N$ method, is widely adopted. It discretizes the angular variables and solves the RTE along discrete directions~\cite{lewis1984computational,adams2002fast}. 
Spherical harmonics methods, by contrast, offer the advantage of rotational invariance and have been extensively applied to radiative transfer problems~\cite{marshak1947note,evans1998spherical}.
Given that the characteristic scale parameter (e.g., Knudsen number) in the RTE can span regimes from kinetic to diffusive, it is essential for numerical methods to handle multiscale behavior robustly. Two major classes of approaches have been proposed to address this challenge: domain decomposition-based methods and asymptotic-preserving (AP) schemes. Domain decomposition-based methods partition the computational domain into subregions, in which different governing equations are solved and coupled through interface conditions~\cite{golse2003domain}. AP schemes, on the other hand, are designed to maintain uniform stability and accuracy across all scale regimes~\cite{jin2010review}. 
For stochastic approaches, the direct simulation Monte Carlo (DSMC) method has been widely used for RTE computations~\cite{alexander1997direct}. Numerous other computational techniques have also been proposed, addressing aspects such as implicit time integration, high-order accuracy, and efficient handling of multiscale effects~\cite{frank2007approximate,li2017implicit,shi2023efficient,xiong2022high,liu2023implicit,han2014two}. 

Recently, there has been growing interest in leveraging deep neural network methods to solve RTEs~\cite{mishra2021physics,chen2022,lu2022solving,jin2023asymptotic,wuAPNNv2,wu2023asymptotic,li2022model,li2025macroscopic,chen2025micro,wu2026}. 
When dealing with RTEs, the vanilla physics-informed neural networks (PINNs)~\cite{raissi2019physics} often suffer from numerical instability arising from the presence of small-scale features~\cite{jin2023asymptotic}.
A key strategy for addressing the multiscale nature of RTEs lies in the careful design of a uniformly accurate loss function that captures the macroscopic limiting behavior. This approach, known as asymptotic-preserving neural networks (APNNs)~\cite{jin2023asymptotic}, is illustrated in Fig.~\ref{fig:apnn}.
\begin{figure}[htbp]
    \centering
    \includegraphics[width=0.5\textwidth]{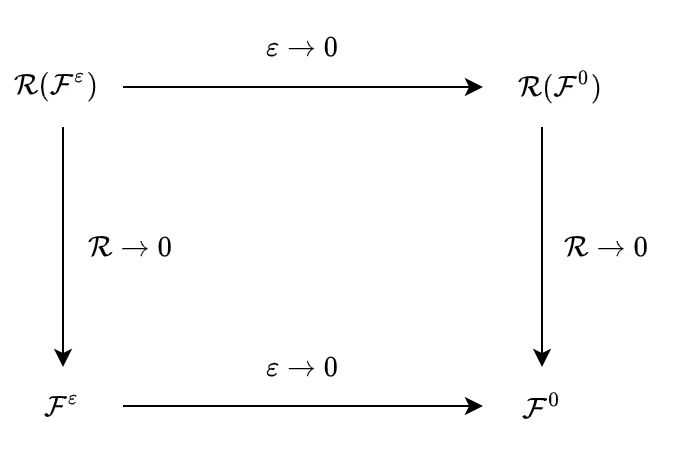}
    \caption{Schematic diagram of the asymptotic-preserving neural networks  (adapted from~\cite{jin2023asymptotic}). 
        Assume $\mathcal{F^{\varepsilon}}$ is the multi-scale model that depends on the
        scale parameter $\varepsilon$ and $\mathcal{F}^{0}$ is the corresponding asymptotic limit model as $\varepsilon \to 0$. Define $\mathcal{R}(\mathcal{F^{\varepsilon}})$ as the neural network-based least-squares formulation of the model $\mathcal{F^{\varepsilon}}$. If $\mathcal{R}(\mathcal{F^{\varepsilon}})$ converges to 
        $\mathcal{R}(\mathcal{F}^{0})$ as $\varepsilon \to 0$, and this limit is precisely the least-squares formulation of the limit model $\mathcal{F}^{0}$, then
        the method is called asymptotic-preserving.
    }
    \label{fig:apnn}
\end{figure}
The APNN framework has been further generalized to address a wide range of problems, such as hyperbolic systems~\cite{bertaglia2022asymptotic2}, the semiconductor Boltzmann equation~\cite{liu2025asymptotic}, the Vlasov–Poisson–Fokker–Planck equation~\cite{jin2024asymptotic,zhang2025ap}, and the Boltzmann equation~\cite{chen2025structure}, thereby providing compelling evidence of its robustness and versatility.

While numerical methods have shown remarkable performance in multiscale kinetic problems, the lack of rigorous analytical justification underscores the need for a deeper theoretical understanding -- this serves as the primary motivation for our study. To the best of our knowledge, existing theoretical analyses have predominantly been developed within the micro–macro decomposition framework.
For instance, \cite{lu2022solving} established a uniform error estimate with respect to the Knudsen number for the steady RTE, whereas \cite{abdo2024error} provided a rigorous proof of the AP property of APNNs applied to the Boltzmann equation. 
Other works, such as \cite{chen2022} and \cite{li2022model}, studied the linear transport equation and gray RTEs, though their analyses are not uniform across regimes. More recently, \cite{chen2025structure} established rigorous statistical estimates for the multiscale control variate method, further advancing the theoretical understanding of such multiscale formulations.

While these efforts have significantly deepened insights into the micro–macro paradigm, our work instead draws inspiration from the parity decomposition framework.
We introduce a novel multiscale parity decomposition and establish a uniform error estimate for the approximate solution with respect to $\varepsilon$. A key refinement in our formulation is the introduction of a new quantity of $\mathcal{O}(\varepsilon^2)$ -- the deviation from local equilibrium, defined in the one-dimensional case as $\varepsilon^2 w = r - \rho$ (see Section~\ref{Sec: One-dimensional method}). The two-dimensional extension involves additional subtleties, which are discussed in detail in Section~\ref{Sec: Two-dimensional method}.
This refinement leads to a fundamental shift in the structure of our error analysis, where the central analytical object becomes the error in the refined variable set $(\rho, w, j)$.

\subsection{Preliminaries}
Consider the time-dependent RTE. Let $f(t,\boldsymbol{x},\boldsymbol{v})$ be the probability density distribution for particles at space point $\boldsymbol{x}\in\mathcal{D}\subset \mathbb{R}^d$, time $t$, and travel in direction $\boldsymbol{v}\in\Omega\subset\mathbb{S}^d$, with $\int_\Omega\mathrm{d}\boldsymbol{v}=S$. Here $\Omega$ is symmetric in $\boldsymbol{v}$, meaning that $\int_\Omega g(\boldsymbol{v})\mathrm{d}\boldsymbol{v}=0$ for any function $g$ odd in $\boldsymbol{v}$. Then $f$ solves the RTE
\begin{equation}\label{general linear transport equation}
    \varepsilon\partial_t f
    +
    \boldsymbol{v}\cdot\nabla_{\bx} f
    =
    \frac{1}{\varepsilon}\Big(\frac{\sigma_S}{S}\int_\Omega f\mathrm{d}\boldsymbol{v}'-\sigma f\Big)+\varepsilon Q,
\end{equation}
where $\sigma=\sigma(\boldsymbol{x})$ is the total transport coefficient, $\sigma_S=\sigma_S(\boldsymbol{x})$ is the scattering coefficient, $Q=Q(\boldsymbol{x})$ is the source term. In this model, the parameter $\varepsilon > 0$ is the Knudsen number, which denotes the ratio of the mean free path to a characteristic length. In particular, $\varepsilon\sim \mathcal{O}(1)$ refers to kinetic regime, and $\varepsilon\ll 1$ corresponds to the diffusive regime. Typically,
\begin{equation}
    \sigma_S
    =
    \sigma-\varepsilon^2\sigma_A,
\end{equation}
where $\sigma_A=\sigma_A(\boldsymbol{x})$ is the absorption coefficient. Such an equation arises in neutron transport, wave propagation in random media. In all these applications, the scaling appeared in \eqref{general linear transport equation}, and gives rise to a diffusion equation as $\varepsilon\rightarrow0$, which is
\begin{equation}\label{eq:diffusion-limit}
    \partial_t \rho = D\nabla_{\bx} \cdot \left ( \frac{1}{\sigma} \nabla_{\bx} \rho \right ) - \sigma_A\,\rho + Q\,,\quad \rho= \frac{1}{S}\int_\Omega f \mathrm{d}\boldsymbol{v}.
 \end{equation}
 \begin{remark}
   Different diffusion coefficient $D$ in \eqref{eq:diffusion-limit} appears for different collision operators. For example, $D = 1/3$ in one-dimensional slab geometry as shown in Section \ref{Sec: One-dimensional method}, and $D = 1/2$ when $\Omega$ is a unit sphere in two dimensions, as will be
discussed in Section \ref{Sec: Two-dimensional method}.  
\end{remark}

In the next sections, for the sake of clarity, we will derive the frameworks for one-dimensional and two-dimensional equations in the simplified situation where $Q=0,\ \sigma_A=0$, and $\sigma=\sigma_S=1$. The extension to the general case does not present additional difficulties and only needs some smooth assumptions. 

\subsection{Main results}
Our main theoretical result is a uniform error estimate for the proposed multiscale parity decomposition framework. 
\begin{theorem}
Let $\mathcal{R}_{\mathrm{APNN}}^\varepsilon$ represent the physics-informed loss based on the residual of the multiscale parity-decomposition system, 
and let $\mathcal{R}_{\mathrm{total}}^\varepsilon$
denote the error between the numerical solution and the exact solution. Then it follows that
\begin{equation*}
  \mathcal{R}_{\mathrm{total}}^\varepsilon
  =
  \mathcal{O}\Big(\sqrt[3]{\big(\mathcal{R}^\varepsilon_{\mathrm{APNN}}\big)^{2}}\Big), \quad \mbox{uniformly in }
  \varepsilon.
\end{equation*}
\end{theorem}
The results conclusively demonstrate that the proposed multiscale parity decomposition framework converges uniformly in the Knudsen number $\varepsilon$.

In fact, as shown in Theorems \ref{Thm: convergence of the APNN solution-classical} and \ref{Thm: convergence of the APNN solution-classical 2D}, a direct numerical analysis gives the error estimate:
\begin{equation*}
    \mathcal{R}_{\mathrm{total}}^\varepsilon
    =
    \mathcal{O}\Big(\frac{\mathcal{R}^\varepsilon_{\mathrm{APNN}}}{\varepsilon}\Big):=\mathcal{E}_1.
\end{equation*}
However, this result exhibits a scaling of $\varepsilon^{-1}$. which becomes unbounded as $\varepsilon\rightarrow 0$. In contrast, the multiscale parity decomposition framework enables a refined analysis, leading to the estimate:
\begin{equation*}
    \mathcal{R}_{\mathrm{total}}^\varepsilon
    =
    \mathcal{O}\big(\mathcal{R}^\varepsilon_{\mathrm{APNN}}+\varepsilon^2\big):=\mathcal{E}_2,
\end{equation*}
established in Theorems \ref{Thm: convergence of the APNN solution} and \ref{Thm: convergence of the APNN solution-2D}.
 Both $\mathcal{E}_1$ and $\mathcal{E}_2$ are mathematically valid and hold simultaneously in this framework. A combination of the two estimates leads to a composite bound:
\begin{equation*}
    \mathcal{R}_{\mathrm{total}}^\varepsilon\leq
    \min\{\mathcal{E}_1,\mathcal{E}_2\},
\end{equation*}
which achieves its upper bound at $\varepsilon=\mathcal{O}\big(\sqrt[3]{\mathcal{R}^\varepsilon_{\mathrm{APNN}}}\big)$, as illustrated in Fig. \ref{fig:aprfm}. Therefore, we obtain the uniform convergence result.

\begin{figure}[htbp]
    \centering
    \includegraphics[width=0.5\textwidth]{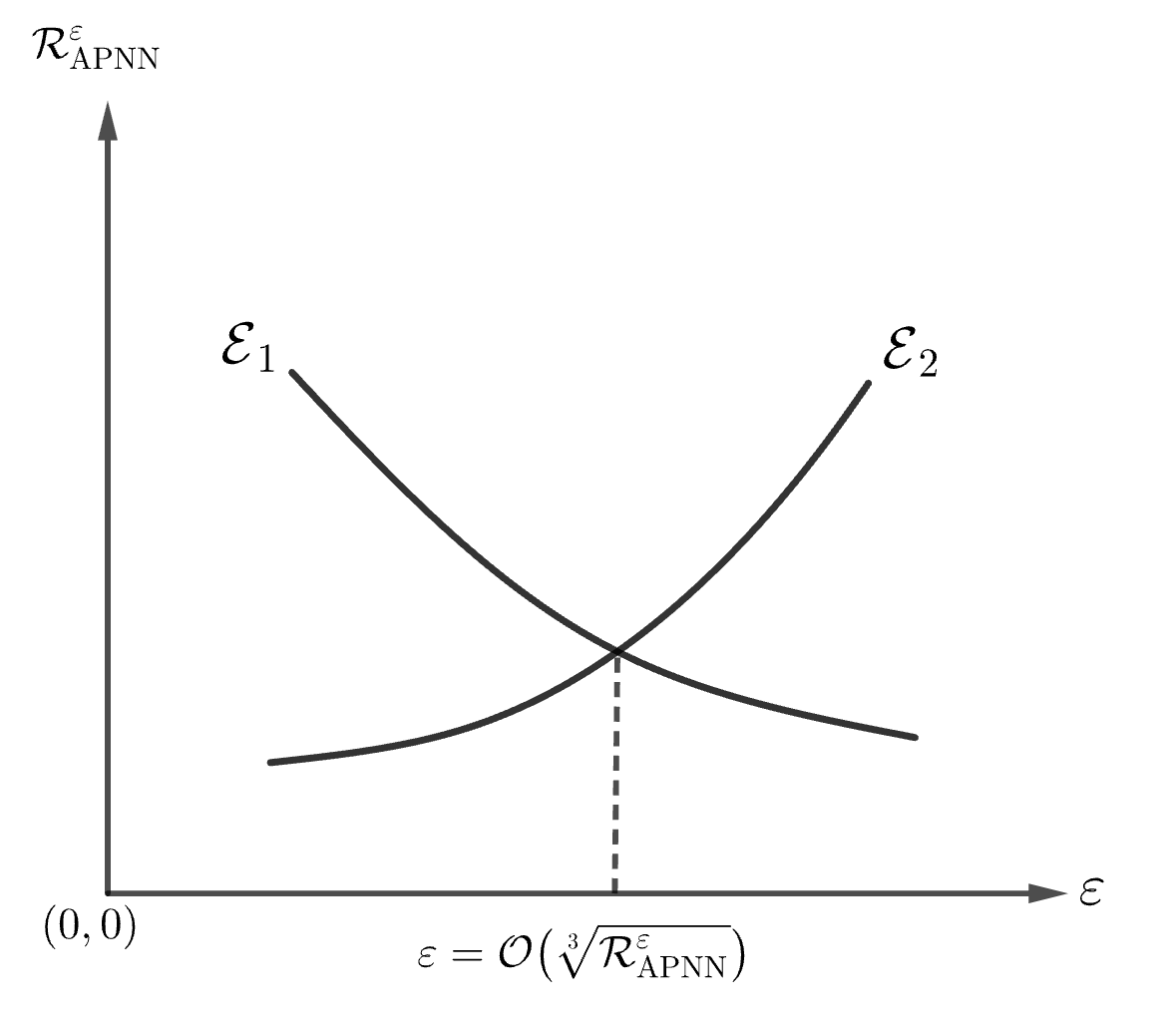}
    \caption{Illustration of uniform convergence of the AP framework.
    }
    \label{fig:aprfm}
\end{figure}

The rest of this paper is organized as follows. Section \ref{Sec: One-dimensional method} details the one-dimensional multiscale parity decomposition framework and the corresponding uniform error estimates. Section \ref{Sec: Two-dimensional method} extends this formulation to two-dimensional problems. Section \ref{Sec: Structure-preserving mechanism} discusses the structure-preserving mechanisms of our AP framework. 
Numerical results are provided in Section \ref{Sec: Numerical experiments}, and conclusions are given in Section \ref{Sec: Conclusions}.

\section{One-dimensional method}\label{Sec: One-dimensional method}
In this section, we consider the one-dimensional RTE in the context of diffusive scaling,  given by 
\begin{equation}\label{linear transport equation}
    \varepsilon\partial_t f
    +
    v\partial_x f
    =
    \frac{1}{\varepsilon}
    \big(\frac{1}{2}\int_{-1}^1 f\, \mathrm{d}v'-f\big), 
\end{equation}
where 
$x\in\mathcal{D}:=(x_L,x_R)$ with velcoity $v\in \Omega:=[-1,1]$. 
We consider the periodic boundary condition
\begin{equation}
    f(t,x_L,v)
    =
    f(t,x_R,v).
\end{equation}
Furthermore,  assume that the initial condition is given as a function of $x$ and $v$:
\begin{equation}
    f(0,x,v):=f_{\mathrm{IC}}(x,v).
\end{equation}

\subsection{Multiscale parity decomposition method}
We next demonstrate the multiscale parity decomposition technique and derive the APNN system for the RTE. Denote $\Omega^+=[0,1]$, and define the even and odd parities for $v\in\Omega^+$ as 
\begin{equation}\label{Def: even- and odd-parities}
    \left\{
    \begin{aligned}
        r(t,x,v)
        =&
        [f(t,x,v)+f(t,x,-v)] / 2, \\
        j(t,x,v)
        =&
        [f(t,x,v)-f(t,x,-v)]/ 2\varepsilon.
    \end{aligned}
    \right.
\end{equation}
Inserting \eqref{Def: even- and odd-parities} into \eqref{linear transport equation}, one gets the system of equations as follows:
\begin{equation}\label{Def: r,j equation}
    \left\{
    \begin{aligned}
        &\varepsilon^2\partial_t r
        +\varepsilon^2v\partial_x j
        =\rho-r,\\
        &\varepsilon^2 \partial_t j+v\partial_x r
        =-j,
    \end{aligned}
    \right.
\end{equation}
where $\rho=\langle r\rangle:=\int_0^1 r(t,x,v)\mathrm{d}v$. To establish a consistent error analysis with respect to $\varepsilon$,  
we introduce $\rho$ as a mediator between $r$ and $j$ within this system.
Integrating the first equation of  \eqref{Def: r,j equation} over $v$ yields the following equation for $\rho$:
\begin{equation}\label{Def: rho equation1}
    \partial_t\rho + \langle v\partial_x j \rangle=0.
\end{equation}
Furthermore, we define $r=\rho +\varepsilon^2 w$, then subtract the first equation of \eqref{Def: r,j equation} by \eqref{Def: rho equation1} to deduce 
\begin{equation}\label{Def: w equation}
    \varepsilon^2\partial_t w
            +w+v\partial_x j-\langle v\partial_x j\rangle=0.
\end{equation}

Therefore, we employ a multiscale parity decomposition of the distribution function $f(t,x,v)$, explicitly separating the macroscopic density $\rho(t, x)$ at $\mathcal{O}(1)$, the odd term $j(t, x, v)$ at $\mathcal{O}(\varepsilon)$, and higher-order corrections $w(t, x, v)$ at $\mathcal{O}(\varepsilon^2)$, each with definite parity properties, i.e.,
\begin{equation}\label{Def: scale decomposition}
f(t, x, v) = \underbrace{\rho(t, x)}_{\mathcal{O}(1)} + \underbrace{\varepsilon j(t, x, v)}_{\mathcal{O}(\varepsilon)} + \underbrace{\varepsilon^2 w(t, x, v)}_{\mathcal{O}(\varepsilon^2)}.
\end{equation}

To summarize, equations \eqref{Def: r,j equation}, \eqref{Def: rho equation1} and \eqref{Def: w equation} constitute the multiscale parity decomposition framework of  \eqref{linear transport equation} to be used for the APNN:
\begin{equation}\label{Def: APNN system}
    \left\{
\begin{aligned}
    &\partial_t\rho + \langle v\partial_x j \rangle=0,\\
    &\varepsilon^2\partial_t j+v\partial_x \rho + \varepsilon^2 v\partial_x w + j = 0,\\
      & \varepsilon^2\partial_t w + w + v\partial_x j - \langle v \partial_x j \rangle = 0, 
\end{aligned}
    \right.
\end{equation}
with the constraint $\langle w \rangle = 0$.
When $\varepsilon\rightarrow 0$, the above equation formally approaches
\begin{equation*}
    \left\{
\begin{aligned}
    &\partial_t\rho + \langle v\partial_x j \rangle=0,\\
    &v\partial_x \rho+j=0,\\
    &w+v\partial_x j-\langle v\partial_x j\rangle=0.
\end{aligned}
    \right.
\end{equation*}
Substituting the second equation into the first and third equations to deduce
\begin{equation}\label{Def: limit equation}
\left\{
\begin{aligned}
   & w=(v^2-\frac{1}{3})\partial_{xx}\rho,\\
    &\partial_t\rho-\frac{1}{3}\partial_{xx}\rho=0.
\end{aligned}
    \right.
\end{equation}

\subsection{Loss function}
Let $\rho_\theta, j_\theta, w_\theta$ be the neural network approximations to $\rho, j, w$ of the \eqref{Def: APNN system}, 
with $\theta$ denoting the trainable parameters of the network.
For the APNN method, we propose the physics-informed loss based on the residual for the multiscale parity decomposition system \eqref{Def: APNN system} as the loss function:
\begin{equation*}
    \mathcal{R}^\varepsilon_{\mathrm{APNN}}
    =
    \mathcal{R}_{\mathrm{residual}}^\varepsilon
    +
    \mathcal{R}_{\mathrm{initial}}^\varepsilon
    +
    \mathcal{R}_{\mathrm{boundary}}^\varepsilon,
\end{equation*}
where $\mathcal{R}_{\mathrm{residual}}^\varepsilon$,  $\mathcal{R}_{\mathrm{initial}}^\varepsilon$ and $\mathcal{R}_{\mathrm{boundary}}^\varepsilon$ are defined as follows
\begin{equation*}
\begin{aligned}
    \mathcal{R}_{\mathrm{residual}}^\varepsilon
    =&
    \frac{1}{|\mathcal{T}\times\mathcal{D}|}
    \int_\mathcal{T}\int_\mathcal{D}
|\partial_t\rho_\theta + \langle v\partial_x j_\theta \rangle|^2\,
    \mathrm{d}x\mathrm{d}t\\
    &+
\frac{1}{|\mathcal{T}\times\mathcal{D}\times\Omega^+|}
    \int_\mathcal{T}\int_\mathcal{D}\int_{\Omega^+}
|\varepsilon^2\partial_t j_\theta
+
v\partial_x \rho_\theta
+
\varepsilon^2 v\partial_x w_\theta
        +
    j_\theta|^2\,
     \mathrm{d}v\mathrm{d}x\mathrm{d}t\\
   & +
    \frac{1}{|\mathcal{T}\times\mathcal{D}\times\Omega^+|}
    \int_\mathcal{T}\int_\mathcal{D}\int_{\Omega^+}
|\varepsilon^2\partial_t w_\theta
            +w_\theta+v\partial_x j_\theta-\langle v\partial_x j_\theta\rangle|^2\,
    \mathrm{d}v\mathrm{d}x\mathrm{d}t,
    \end{aligned}
\end{equation*}
\begin{equation*}
    \begin{aligned}
       \mathcal{R}_{\mathrm{initial}}^\varepsilon
       =&
       \frac{\lambda_1}{|\mathcal{D}|}\int_{\mathcal{D}}
       |\rho_\theta-\rho_{\rm IC}|^2\,
       \mathrm{d}x\\
       &
       +
       \frac{\lambda_2}{|\mathcal{D}\times\Omega^+|}
       \int_\mathcal{D}\int_{\Omega^+}
       \varepsilon^2|j_\theta-j_{\rm IC}|^2
       +
\varepsilon^4|w_\theta-w_{\rm IC}|^2\,
       \mathrm{d}v\mathrm{d}x,
    \end{aligned}
\end{equation*}
and 
\begin{equation*}
    \begin{aligned}
       \mathcal{R}_{\mathrm{boundary}}^\varepsilon
       =&
       \frac{\lambda_3}{|\mathcal{T}\times\partial\mathcal{D}|}
       \int_\mathcal{T}\int_{\partial\mathcal{D}}
       |\rho_\theta-\rho_{\rm BC}|^2\,
       \mathrm{d}s\mathrm{d}t\\
       &+
       \frac{\lambda_4}{|\mathcal{T}\times\partial\mathcal{D}\times\Omega^+|}
       \int_\mathcal{T}\int_{\partial\mathcal{D}}\int_{\Omega^+}
|w_\theta-w_{\rm BC}|^2+|j_\theta-j_{\rm BC}|^2\,
       \mathrm{d}v\mathrm{d}s\mathrm{d}t.
    \end{aligned}
\end{equation*}
Here, $\lambda_i\,(i=1,\cdots,4)$ are hyperparameters as penalty terms, and $\mathcal{T},\mathcal{D},\Omega^+$ represent the
bounded domains of time, space, and velocity space, respectively.

The AP property of this loss can be carried out by considering its behavior for
$\varepsilon$ small. One may only need to focus on the term $\mathcal{R}_{\mathrm{residual}}^\varepsilon$. 
 As $\varepsilon\rightarrow 0$, it yields that $\mathcal{R}_{\mathrm{residual}}^\varepsilon$ converges to 
\begin{equation*}
    \begin{aligned}
       \mathcal{R}_{\mathrm{residual}}^0 
       :=&
       \frac{1}{|\mathcal{T}\times\mathcal{D}|}
    \int_\mathcal{T}\int_\mathcal{D}
|\partial_t\rho_\theta + \langle v\partial_x j_\theta \rangle|^2\,
    \mathrm{d}x\mathrm{d}t\\
    &+
    \frac{1}{|\mathcal{T}\times\mathcal{D}\times\Omega^+|}
    \int_\mathcal{T}\int_\mathcal{D}\int_{\Omega^+}
|v\partial_x \rho_\theta
        +
    j_\theta|^2\,
     \mathrm{d}v\mathrm{d}x\mathrm{d}t\\
     &+
    \frac{1}{|\mathcal{T}\times\mathcal{D}\times\Omega^+|}
    \int_\mathcal{T}\int_\mathcal{D}\int_{\Omega^+}
|w_\theta+v\partial_x j_\theta-\langle v\partial_x j_\theta\rangle|^2\,
    \mathrm{d}v\mathrm{d}x\mathrm{d}t,\\
    \end{aligned}
\end{equation*}
which is exactly the loss for the limiting equations \eqref{Def: limit equation}. Thus, our designed loss function for APNN satisfies the AP property.

\subsection{Convergence of the loss function}

As mentioned above, we build the loss function based on the residual of the multiscale parity decomposition system \eqref{Def: APNN system}. Therefore, our approximated solutions of the neural network are $\rho,j$ and $w$, instead of $f$ if we apply PINN to the model \eqref{linear transport equation}.

Let us first review an important result on the existence of the approximated neural network solution, namely the Universal Approximation Theorem.
\begin{lemma}[\cite{hu2023higher} Lemma A.1]\label{Lemma: Universal Approximation Theorem}
Let $\Omega=\prod_{i=1}^d[a_i,b_i]$. Suppose $f\in H^m(\Omega)$. Let $N>5$ be an integer. Then there exists a $\tanh$ neural network $\hat{f}^N$ with two hidden layers, such that for any $k\in\{0,1,\cdots, m-1\}$,
\begin{equation*}
   \|f-\hat{f}^N\|_{H^k(\Omega)}
    \leq
    C_{k,d,f,\Omega}(1+\ln^k N)N^{-m+k}.
\end{equation*}
Here, the width of the $\tanh$ neural network depends on $N$.
\end{lemma}

Based on this approximation property, we can establish the convergence of the APNN loss function.
\begin{theorem}\label{Thm: convergence of the loss fuction}
    Let $\rho\in H^2(\mathcal{T}\times\mathcal{D}),\ j,w\in H^2(\mathcal{T}\times\mathcal{D}\times\Omega^+)$ be the analytic solution to the multiscale parity decomposition system \eqref{Def: APNN system}, and $N>5$ be an integer. Then there exist APNN solutions $\rho_\theta=\hat{\rho}^N,j_\theta=\hat{j}^N,w_\theta=\hat{w}^N$ such that
   \begin{equation*}
        \mathcal{R}^\varepsilon_{\mathrm{APNN}}\rightarrow 0,\quad \mbox{as } N\rightarrow \infty.
    \end{equation*}
\end{theorem}
\begin{proof}
    Since $\rho\in H^2(\mathcal{T}\times\mathcal{D}),\ j,w\in H^2(\mathcal{T}\times\mathcal{D}\times\Omega^+)$, then there exist neural network solutions $\rho_\theta=\hat{\rho}^N,j_\theta=\hat{j}^N,w_\theta=\hat{w}^N$, such that
   \begin{equation}
        \|\rho-\rho_\theta\|_{H^1(\mathcal{T}\times\mathcal{D})}
        \leq
        C(1+\ln N)N^{-1},
    \end{equation}
    and 
    \begin{equation}\label{approximate of j w}
        \|u-u_\theta\|_{H^1(\mathcal{T}\times\mathcal{D}\times\Omega^+)}
        \leq
      C(1+\ln N)N^{-1},
   \end{equation}
    where $u=j,w$. 
       For brevity of notation, we define
    \begin{equation*}
        d_\theta^{(1)}(t,x)
        :=
        \partial_t\rho_\theta
        +
        \langle v\partial_x j_\theta\rangle.
    \end{equation*}
    Since $\partial_t\rho+\langle v\partial_x j\rangle=0$, subtract $d_\theta^{(1)}(t,x)$ by the above equation, then
    \begin{equation}\label{d1}
        d_\theta^{(1)}(t,x)
        =
        \partial_t(\rho_\theta-\rho)
        +
        \langle v\partial_x(j_\theta-j)\rangle.
    \end{equation}
By integrating $|d_\theta^{(1)}|^2$ over $\mathcal{T}\times\mathcal{D}$, it yields that
\begin{equation}\label{estimate of loss 1}
\begin{aligned}
    &\|\partial_t(\rho_\theta-\rho)
        +
        \langle v\partial_x(j_\theta-j)\rangle\|_{L^2(\mathcal{T}\times\mathcal{D})}^2\\
        &\leq
        \|\partial_t(\rho-\rho_\theta)\|_{L^2(\mathcal{T}\times\mathcal{D})}^2
        +
        \|\partial_t(j-j_\theta)\|_{L^2(\mathcal{T}\times\mathcal{D}\times\Omega^+)}^2
        \leq
       C(1+\ln N)^2N^{-2},
       \end{aligned}
\end{equation}
 which corresponds to the estimate of the first term in the loss function $\mathcal{R}^\varepsilon_{\mathrm{residual}}$. For the other terms, denote
\begin{equation}\label{d2}
    d_\theta^{(2)}(t,x)
        :=
        \varepsilon^2\partial_t j_\theta
        +
       j_\theta
        +
        v\partial_x \rho_\theta
        +
        \varepsilon^2 v\partial_x w_\theta,
\end{equation}
and
\begin{equation}\label{d3}
    d_\theta^{(3)}(t,x)
    :=
    \varepsilon^2\partial_t w_\theta
            +w_\theta+v\partial_x j_\theta-\langle v\partial_x j_\theta\rangle.
\end{equation}
Since $\rho,j$, and $w$ are the solutions of the linear system \eqref{Def: APNN system}, we subtract $d_\theta^{(2)}$ and $d_\theta^{(3)}$ from the second and last equations of \eqref{Def: APNN system}, respectively, to obtain
\begin{equation*}
   \left\{
\begin{aligned}
     d_\theta^{(2)}(t,x)
        =&
        \varepsilon^2\partial_t(j_\theta-j)
        +
        (j_\theta-j)
        +
        v\partial_x(\rho_\theta-\rho)
        +
         \varepsilon^2 v\partial_x(w_\theta-w),\\
        d_\theta^{(3)}(t,x)
    =&
    \varepsilon^2\partial_t (w_\theta-w)
            +
            (w_\theta-w)
            +
            v\partial_x (j_\theta-j)
            -
            \langle v\partial_x (j_\theta-j)\rangle.
\end{aligned}
   \right.
\end{equation*}
Integrating $\big|d_\theta^{(2)}\big|^2$ and $\big|d_\theta^{(3)}\big|^2$ over $\mathcal{T}\times\mathcal{D}\times\Omega^+$, one gets
\begin{equation*}
    \begin{aligned} 
   \|\varepsilon^2\partial_t(j_\theta-j)
        +&
        (j_\theta-j)
        +
         v\partial_x(\rho_\theta-\rho)
        +
         \varepsilon^2 v\partial_x(w_\theta-w)\|_{L^2(\mathcal{T}\times\mathcal{D}\times\Omega^+)}^2\\
        \leq&
        \varepsilon^2\|\partial_t(j_\theta-j)\|_{L^2{(\mathcal{T}\times\mathcal{D}\times\Omega^+)}}^2
        +
        \|j_\theta-j\|_{L^2{(\mathcal{T}\times\mathcal{D}\times\Omega^+)}}^2\\
       & +
        \|\partial_x(\rho_\theta-\rho)\|_{L^2{(\mathcal{T}\times\mathcal{D}}}^2
        +
         \varepsilon^2\|\partial_x(w_\theta-w)\|_{L^2{(\mathcal{T}\times\mathcal{D}\times\Omega^+)}}^2,
    \end{aligned}
\end{equation*}
and 
\begin{equation*}
    \begin{aligned}
   \|\varepsilon^2&\partial_t (w_\theta-w)
            +
            (w_\theta-w)
            +
            v\partial_x (j_\theta-j)
            -
            \langle v\partial_x (j_\theta-j)\rangle\|_{L^2(\mathcal{T}\times\mathcal{D}\times\Omega^+)}^2\\
            \leq&
        \varepsilon^2\|\partial_t(w_\theta-w)\|_{L^2{(\mathcal{T}\times\mathcal{D}\times\Omega^+)}}^2
        +
        \|w_\theta-w\|_{L^2{(\mathcal{T}\times\mathcal{D}\times\Omega^+)}}^2
        +
        2\|\partial_x(j_\theta-j)\|_{L^2{(\mathcal{T}\times\mathcal{D}\times\Omega^+)}}^2.
    \end{aligned}
\end{equation*}
Thus from \eqref{approximate of j w}, we obtain
\begin{equation}\label{estimate of loss 2}
    \|d_\theta^{(2)}\|_{L^2(\mathcal{T}\times\mathcal{D}\times\Omega^+)}^2
    +
    \|d_\theta^{(3)}\|_{L^2(\mathcal{T}\times\mathcal{D}\times\Omega^+)}^2
    \leq
       C(1+\ln N)^2N^{-2}.
\end{equation}
Combining the estimates \eqref{estimate of loss 1} and \eqref{estimate of loss 2} to get
\begin{equation}\label{estimate of residual}
   \mathcal{R}_{\mathrm{residual}}^\varepsilon
    \leq
       C(1+\ln N)^2N^{-2}.
\end{equation}
For the first term of the initial residual, using the trace inequality to get
\begin{equation*}
\begin{aligned}
    \|\rho_\theta-\rho_{\rm IC}\|_{L^2(\mathcal{D})}
    \leq&
    \|\rho_\theta-\rho\|_{L^2(\partial (\mathcal{T}\times\mathcal{D}))}\\
    \leq&
    \|\rho_\theta-\rho\|_{H^1( \mathcal{T}\times\mathcal{D})}
    \leq
    C(1+\ln N)N^{-1}.
    \end{aligned}
\end{equation*}
Similarly, we can deduce that
\begin{equation*}
\begin{aligned}
    \varepsilon\|j_\theta-j_{\rm IC}\|_{L^2(\mathcal{D}\times\Omega^+)}
    +
    \varepsilon^2\|w_\theta-w_{\rm IC}\|_{L^2(\mathcal{D}\times\Omega^+)}
    \leq
    C(1+\ln N)N^{-1}.
    \end{aligned}
\end{equation*}
Then it yields that
\begin{equation}\label{estimate of initial}
    \mathcal{R}_{\mathrm{initial}}^\varepsilon
    \leq
    C(1+\ln N)^2N^{-2}.
\end{equation}
Moreover, we can also estimate the boundary residual using trace inequalities to get
\begin{equation}\label{estimate of boundary}
    \mathcal{R}_{\mathrm{boundary}}^\varepsilon
    \leq
    C(1+\ln N)^2N^{-2}.
\end{equation}
   Combining the estimates \eqref{estimate of residual}, \eqref{estimate of initial}, and \eqref{estimate of boundary},  we conclude that
    \begin{equation*}
        \mathcal{R}_{\mathrm{APNN}}^\varepsilon
        =
        \mathcal{R}_{\mathrm{residual}}^\varepsilon
        +
        \mathcal{R}_{\mathrm{initial}}^\varepsilon
        +
        \mathcal{R}_{\mathrm{boundary}}^\varepsilon
        \leq
        C(1+\ln N)^2N^{-2}.
    \end{equation*}
    Therefore, this result is proven.

\end{proof}

\subsection{Uniform error analysis}
Next, we prove that the APNN converges to the analytic solution of the RTE \eqref{linear transport equation} uniformly with respect to $\varepsilon$. To quantify the convergence behavior, we define the error between analytic solutions for system \eqref{Def: APNN system} and APNN solutions as $\mathcal{R}_{\mathrm{total}}^\varepsilon$ in the following form:
    \begin{equation}\label{Def: total error}
    \begin{aligned}
        \mathcal{R}_{\mathrm{total}}^\varepsilon
        =&
        \|f-f_\theta\|_{L^2(\mathcal{T}\times\mathcal{D}\times\Omega^+)}^2.
        \end{aligned}
    \end{equation}
The following theorem establishes a bound for this total error using direct techniques.

\begin{theorem}\label{Thm: convergence of the APNN solution-classical}
    Let $\rho\in H^1(\mathcal{T}\times\mathcal{D}),\ j,\ w\in H^1(\mathcal{T}\times\mathcal{D}\times\Omega^+)$ be the analytic solution to the multiscale parity decomposition  system \eqref{Def: APNN system}. Then
    \begin{equation*}
        \mathcal{R}_{\mathrm{total}}^\varepsilon
        \leq
        \frac{C}{\varepsilon}\mathcal{R}^\varepsilon_{\mathrm{APNN}}.
    \end{equation*}
\end{theorem}
\begin{proof}
Recall the defintion of $d_\theta^{(1)},d_\theta^{(2)},d_\theta^{(3)}$ in \eqref{d1}, \eqref{d2}, and \eqref{d3}. Denote $\tilde{\rho}=\rho_\theta-\rho$, $\tilde{j}=j_\theta-j$ and $\tilde{w}=w_\theta-w$, then we have
\begin{equation*}
            \left\{
        \begin{aligned}
   & \partial_t\tilde{\rho}
    +
    \langle v\partial_x\tilde{j} \rangle
    =
    d_\theta^{(1)}(t,x),\\
    &\varepsilon^2\partial_t\tilde{j}+\tilde{j}
    +
     v\partial_x\tilde{\rho} 
     +
     \varepsilon^2 v\partial_x\tilde{w} 
    =
     d_\theta^{(2)}(t,x),\\
    & \varepsilon^2\partial_t\tilde{w}+\tilde{w}
    +
    v\partial_x\tilde{j}
    -
    \langle v\partial_x\tilde{j} \rangle
    =
     d_\theta^{(3)}(t,x).
      \end{aligned}
            \right.
        \end{equation*}
     Based on the decomposition of $f$ in \eqref{Def: scale decomposition}, we use the notation $\tilde{f}=f-f_\theta$ with 
     \begin{equation}
     \begin{aligned}
     \tilde{f}
     =
     \tilde{\rho}
     +
     \varepsilon\tilde{j}
     +
     \varepsilon^2 \tilde{w},
         \end{aligned}
     \end{equation}
     which is governed by 
\begin{equation}
    \varepsilon^2\partial_t\tilde{f}
    +
    \varepsilon v\partial_x\tilde{f}
    =
    (\tilde{\rho}-\tilde{f})
    +
    \varepsilon^2 d_\theta^{(1)}+ \varepsilon d_\theta^{(2)}+ \varepsilon^2 d_\theta^{(3)}.
\end{equation}
Multiply the above equation by $\tilde{f}$ and integrate in $\mathcal{D}\times \Omega^+$ to get
\begin{equation*}
    \begin{aligned}
       \frac{\varepsilon^2}{2}\frac{\mathrm{d}}{\mathrm{d}t}\|\tilde{f}\|_{L^2(\mathcal{D}\times \Omega^+)}^2 
       +&
       \int_\mathcal{D}\int_{\Omega^+}
\varepsilon v\partial_x\tilde{f}\,\tilde{f}
       \mathrm{d}v\mathrm{d}x\\
       =&
       \int_\mathcal{D}\int_{\Omega^+}
       (\tilde{\rho}-\tilde{f})\tilde{f}
       \mathrm{d}v\mathrm{d}x
       +
       \int_\mathcal{D}\int_{\Omega^+}
       (\varepsilon^2 d_\theta^{(1)}+ \varepsilon d_\theta^{(2)}+ \varepsilon^2 d_\theta^{(3)})\tilde{f}
       \mathrm{d}v\mathrm{d}x.
    \end{aligned}
\end{equation*}
    By integration by parts, we have 
    \begin{equation}
        \int_\mathcal{D}\int_{\Omega^+}
\varepsilon v\partial_x\tilde{f}\,\tilde{f}
       \mathrm{d}v\mathrm{d}x
       =
       0.
    \end{equation}
    Since $\tilde{\rho}=\langle\tilde{f}\rangle$, it follows that
    \begin{equation}
         \int_\mathcal{D}\int_{\Omega^+}
       (\tilde{\rho}-\tilde{f})\tilde{f}
       \mathrm{d}v\mathrm{d}x
       \leq
       0.
    \end{equation}
   In addition, by applying Young's inequality, we obtain
    \begin{equation*}
    \begin{aligned}
    \int_\mathcal{D}&\int_{\Omega^+}
       (\varepsilon^2 d_\theta^{(1)}+ \varepsilon d_\theta^{(2)}+ \varepsilon^2 d_\theta^{(3)})\tilde{f}
       \mathrm{d}v\mathrm{d}x\\
       \leq&
       \varepsilon^2\|\tilde{f}\|_{L^2(\mathcal{D}\times \Omega^+)}^2 
       +
       \varepsilon^2\|d_\theta^{(1)}\|_{L^2(\mathcal{D})}^2 
       +
       \varepsilon\|d_\theta^{(2)}\|_{L^2(\mathcal{D}\times \Omega^+)}^2 
       +
       \varepsilon^2\|d_\theta^{(3)}\|_{L^2(\mathcal{D}\times \Omega^+)}^2 .
       \end{aligned}
        \end{equation*}
      Hence, we can conclude that 
      \begin{equation*}
              \frac{\mathrm{d}}{\mathrm{d}t}\|\tilde{f}\|_{L^2(\mathcal{D}\times \Omega^+)}^2 
              \leq 
              2\|\tilde{f}\|_{L^2(\mathcal{D}\times \Omega^+)}^2 
               +
               2\|d_\theta^{(1)}\|_{L^2(\mathcal{D})}^2 \\
               +
               \frac{2}{\varepsilon}\|d_\theta^{(2)}\|_{L^2(\mathcal{D}\times \Omega^+)}^2 
               +
               2\|d_\theta^{(3)}\|_{L^2(\mathcal{D}\times \Omega^+)}^2 .
      \end{equation*}      
      Therefore, by Gronwall's inequality, we prove this result.
\end{proof}

Unfortunately, the above result indicates that the error may become unbounded as $\varepsilon$ approaches zero. To overcome this limitation and obtain a uniform bound, we proceed with a more refined error analysis.
Recall the decomposition of $f$ in \eqref{Def: scale decomposition}, and apply triangle inequality to get the relation:
\begin{equation*}
    \begin{aligned}
        \mathcal{R}_{\mathrm{total}}^\varepsilon
        \leq
        \|\rho-\rho_\theta\|_{L^2(\mathcal{T}\times\mathcal{D})}^2
        +
        \varepsilon^2\|j-j_\theta\|_{L^2(\mathcal{T}\times\mathcal{D}\times\Omega^+)}^2
        +
        \varepsilon^4\|w-w_\theta\|_{L^2(\mathcal{T}\times\mathcal{D}\times\Omega^+)}^2.
        \end{aligned}
    \end{equation*}
Therefore, instead of estimating the total error directly, we estimate each term on the right-hand side of the above equation.

    \begin{theorem}\label{Thm: convergence of the APNN solution}
    Let $\rho\in H^1(\mathcal{T}\times\mathcal{D}),\ j,\ w\in H^1(\mathcal{T}\times\mathcal{D}\times\Omega^+)$ be the analytic solution to the multiscale parity decomposition system \eqref{Def: APNN system}. Then
    \begin{equation*}
        \mathcal{R}_{\mathrm{total}}^\varepsilon
        \leq
        C(\mathcal{R}^\varepsilon_{\mathrm{APNN}}+\varepsilon^2),
    \end{equation*}
    where the constant $C>0$ is independent of $\varepsilon$.
    
    \end{theorem}

    \begin{proof}
We follow the notation established in the proof of Theorem \ref{Thm: convergence of the APNN solution-classical}. For the sake of clarity, we recall the system of equations:
        \begin{equation}\label{govern equation of rho,j,w}
            \left\{
        \begin{aligned}
   & \partial_t\tilde{\rho}
    +
    \langle v\partial_x\tilde{j} \rangle
    =
    d_\theta^{(1)}(t,x),\\
    &\varepsilon^2\partial_t\tilde{j}+\tilde{j}
    +
     v\partial_x\tilde{\rho} 
     +
    \varepsilon^2 v\partial_x\tilde{w} 
    =
     d_\theta^{(2)}(t,x),\\
    & \varepsilon^2\partial_t\tilde{w}+\tilde{w}
    +
    v\partial_x\tilde{j}
    -
    \langle v\partial_x\tilde{j} \rangle
    =
     d_\theta^{(3)}(t,x).
      \end{aligned}
            \right.
        \end{equation}
        Here, the notations $d_\theta^{(1)},d_\theta^{(2)},d_\theta^{(3)}$ are defined in \eqref{d1}, \eqref{d2}, and \eqref{d3}.
     Multiply $\eqref{govern equation of rho,j,w}_1$ by $\tilde{\rho}$ and integrate over $\mathcal{D}$ to deduce that
    \begin{equation}\label{rho estimate}
        \frac{1}{2}\frac{\mathrm{d}}{\mathrm{d}t}\|\tilde{\rho}\|_{L^2(\mathcal{D})}^2
        +
        \int_\mathcal{D}\int_{\Omega^+}
v\partial_x\tilde{j}\,\tilde{\rho}
        \,\mathrm{d}v\mathrm{d}x
        =
        \int_\mathcal{D}d_\theta^{(1)}(t,x)\tilde{\rho}\,
        \mathrm{d}x.
    \end{equation}
Multiply $\eqref{govern equation of rho,j,w}_2, \eqref{govern equation of rho,j,w}_3$ by $\tilde{j}$, $\varepsilon^2\tilde{w}$, respectively, and integrate over $\mathcal{D}\times\Omega^+$ to get
\begin{equation}\label{j estimate}
\begin{aligned}
    \frac{\varepsilon^2}{2}\frac{\mathrm{d}}{\mathrm{d}t}\|\tilde{j}\|_{L^2(\mathcal{D}\times\Omega^+)}^2
    +&
    \|\tilde{j}\|_{L^2(\mathcal{D}\times\Omega^+)}^2\\
    +&
        \int_\mathcal{D}\int_{\Omega^+}
v\partial_x(\tilde{\rho}+\varepsilon^2 \tilde{w})\,\tilde{j}
        \,\mathrm{d}v\mathrm{d}x
        =
        \int_\mathcal{D}\int_{\Omega^+}
        d_\theta^{(2)}(t,x)\tilde{j}\,
        \mathrm{d}v\mathrm{d}x,
        \end{aligned}
\end{equation}
and
\begin{equation}\label{w estimate}
    \begin{aligned}
       \frac{\varepsilon^4}{2}\frac{\mathrm{d}}{\mathrm{d}t}\|\tilde{w}\|_{L^2(\mathcal{D}\times\Omega^+)}^2
    +&
    \varepsilon^2\|\tilde{w}\|_{L^2(\mathcal{D}\times\Omega^+)}^2
    +
        \varepsilon^2\int_\mathcal{D}\int_{\Omega^+}
v\partial_x\tilde{j}\,\tilde{w}
        \,\mathrm{d}v\mathrm{d}x\\
        =&
        \varepsilon^2\int_\mathcal{D}\int_{\Omega^+}
        d_\theta^{(3)}(t,x)\tilde{w}\,
        \mathrm{d}v\mathrm{d}x
        +
        \varepsilon^2\int_\mathcal{D}\int_{\Omega^+}
\langle v\partial_x\tilde{j} \rangle\tilde{w}\,
        \mathrm{d}v\mathrm{d}x.
    \end{aligned}
\end{equation}
Notice that combining the cross term of the left-hand side in the above three equalities and deriving
\begin{equation}
\begin{aligned}
    \int_\mathcal{D}\int_{\Omega^+}&
v\partial_x\tilde{j}\,\tilde{\rho}
        \,\mathrm{d}v\mathrm{d}x
        +
        \int_\mathcal{D}\int_{\Omega^+}
v\partial_x(\tilde{\rho}+\varepsilon^2 \tilde{w})\,\tilde{j}
        \,\mathrm{d}v\mathrm{d}x
        +
    \varepsilon^2\int_\mathcal{D}\int_{\Omega^+}
v\partial_x\tilde{j}\,\tilde{w}
        \,\mathrm{d}v\mathrm{d}x  \\
        =&
        \int_\mathcal{D}\int_{\Omega^+}
v\partial_x\tilde{j}\,(\tilde{\rho}+\varepsilon^2 \tilde{w})
        \,\mathrm{d}v\mathrm{d}x
        +
        \int_\mathcal{D}\int_{\Omega^+}
v\partial_x(\tilde{\rho}+\varepsilon^2 \tilde{w})\,\tilde{j}
        \,\mathrm{d}v\mathrm{d}x\\
        =&
        \int_\mathcal{D}\int_{\Omega^+}
v\partial_x(\tilde{j}(\tilde{\rho}+\varepsilon^2 \tilde{w}))
        \,\mathrm{d}v\mathrm{d}x
        =
        0,
        \end{aligned}
\end{equation}
where the last equality holds since the term vanishes under periodic boundary conditions. In addition, we estimate the right-hand terms of equations \eqref{rho estimate}, \eqref{j estimate}, and \eqref{w estimate} one by one. 
By Young's inequality, one gets
\begin{equation*}
\begin{aligned}
    \int_\mathcal{D}&d_\theta^{(1)}(t,x)\tilde{\rho}
        \,\mathrm{d}x
        +
 \int_\mathcal{D}\int_{\Omega^+}
        d_\theta^{(2)}(t,x)\tilde{\rho}\,
        \mathrm{d}v\mathrm{d}x
        +
        \varepsilon^2\int_\mathcal{D}\int_{\Omega^+}
        d_\theta^{(3)}(t,x)\tilde{w}\,
        \mathrm{d}v\mathrm{d}x\\
        \leq&\,
        \frac{1}{2}\Big(\|d_\theta^{(1)}(t,x)\|_{L^2(\mathcal{D})}^2
   +
       \|d_\theta^{(2)}(t,x)\|_{L^2(\mathcal{D}\times\Omega^+)}^2
+
 \|d_\theta^{(3)}(t,x)\|_{L^2(\mathcal{D}\times\Omega^+)}^2
        \Big)\\
        &+
        \frac{1}{2}\Big(\|\tilde{\rho}\|_{L^2(\mathcal{D})}^2
        +
        \|\tilde{j}\|_{L^2(\mathcal{D}\times\Omega^+)}^2
        +
        \varepsilon^4\|\tilde{w}\|_{L^2(\mathcal{D}\times\Omega^+)}^2
        \Big),
        \end{aligned}
\end{equation*}
and
\begin{equation*}
    \varepsilon^2\int_\mathcal{D}\int_{\Omega^+}
\langle v\partial_x\tilde{j} \rangle\tilde{w}\,
        \mathrm{d}v\mathrm{d}x
        \leq
        \frac{\varepsilon^2}{2}\|\partial_x \tilde{j}\|_{L^2(\mathcal{D}\times\Omega^+)}
        +
        \frac{\varepsilon^2}{2}\|\tilde{w}\|_{L^2(\mathcal{D}\times\Omega^+)}^2.
\end{equation*}
Therefore, we can conclude that
\begin{equation*}
    \begin{aligned}
        \frac{\mathrm{d}}{\mathrm{d}t}
        \big(\|\tilde{\rho}\|_{L^2(\mathcal{D})}^2
+&
\varepsilon^2\|\tilde{j}\|_{L^2(\mathcal{D}\times\Omega^+)}^2
+
\varepsilon^4\|\tilde{w}\|_{L^2(\mathcal{D}\times\Omega^+)}^2
        \big)
        +
        \|\tilde{j}\|_{L^2(\mathcal{D}\times\Omega^+)}^2
        +
        \varepsilon^2\|\tilde{w}\|_{L^2(\mathcal{D}\times\Omega^+)}^2\\
        \leq&
        \big(\|d_\theta^{(1)}(t,x)\|_{L^2(\mathcal{D})}^2
   +
       \|d_\theta^{(2)}(t,x)\|_{L^2(\mathcal{D}\times\Omega^+)}^2
+
 \|d_\theta^{(3)}(t,x)\|_{L^2(\mathcal{D}\times\Omega^+)}^2\\
 &+
 \varepsilon^2\|\partial_x \tilde{j}\|_{L^2(\mathcal{D}\times\Omega^+)}
        \big)
        +
        \big(\|\tilde{\rho}\|_{L^2(\mathcal{D})}^2
        +
        \varepsilon^4\|\tilde{w}\|_{L^2(\mathcal{D}\times\Omega^+)}^2
        \big).
        \end{aligned}
\end{equation*}
An application of Gronwall's inequality concludes the proof.

\end{proof}

\section{Two-dimensional method}\label{Sec: Two-dimensional method}
Building upon the one-dimensional analysis in the previous section, we now generalize the framework to the two-dimensional case. Consider the two-dimensional RTE:
    \begin{equation}\label{linear transport equation2}
        \varepsilon\partial_t f
        +
    \boldsymbol{v}\cdot\nabla_{\bx} f
    =
    \frac{1}{\varepsilon}\Big(\frac{1}{2\pi}\int_{|\boldsymbol{v}|=1}f\,\mathrm{d}\boldsymbol{v}'
    -f\Big),\quad \boldsymbol{x} = (x_1, x_2) \in \mathcal{D}\subset \mathbb{R}^2,\ |\boldsymbol{v}|=1
    \end{equation}
with $\boldsymbol{v}=(\xi,\eta),\ -1\leq \xi,\eta\leq 1$, $\xi^2+\eta^2=1$. 
The periodic boundary is given by $f(t,\boldsymbol{x},\boldsymbol{v}) = f_{\mathrm{BC}}(\boldsymbol{x},\boldsymbol{v}).$
We assume the initial condition that $f(0,\boldsymbol{x},\boldsymbol{v}) = f_{\mathrm{IC}}(\boldsymbol{x},\boldsymbol{v}).$

\subsection{Multiscale parity decomposition method}  We now introduce the multiscale parity decomposition for the two-dimensional RTE. Similarly to the one-dimensional case, for convenience of representation, define the set of all positive $\xi$ and $\eta$ as
\begin{equation}
    \Omega^+
    =
    \big\{\xi>0,\eta>0\,|\,\xi^2+\eta^2=1,-1\leq \xi,\eta\leq 1\big\}.
\end{equation}
We rewrite \eqref{linear transport equation2} in $\xi,\eta\in \Omega^+$ only:
\begin{equation*}
    \begin{gathered}
        \varepsilon\partial_t f(\xi,\eta)+\xi \partial_x f(\xi,\eta)
        +\eta\partial_y f(\xi,\eta)
        =
        \frac{1}{\varepsilon}\big(\rho-f(\xi,\eta)\big),\\
        \varepsilon\partial_t f(-\xi,-\eta)-\xi \partial_x f(-\xi,-\eta)
        -\eta\partial_y f(-\xi,-\eta)
        =
        \frac{1}{\varepsilon}\big(\rho-f(-\xi,-\eta)\big),\\
        \varepsilon\partial_t f(\xi,-\eta)+\xi \partial_x f(\xi,-\eta)
        -\eta\partial_y f(\xi,-\eta)
        =
        \frac{1}{\varepsilon}\big(\rho-f(\xi,-\eta)\big),\\
        \varepsilon\partial_t f(-\xi,\eta)-\xi \partial_x f(-\xi,\eta)
        +\eta\partial_y f(-\xi,\eta)
        =
        \frac{1}{\varepsilon}\big(\rho-f(-\xi,\eta)\big),\\
    \end{gathered}
\end{equation*}
where 
\begin{equation*}
    \rho=
    \frac{1}{2\pi}\int_{|\boldsymbol{v}|=1}f\,\mathrm{d}\boldsymbol{v}'.
\end{equation*}
Introducing the even and odd parities respect to the the vector $(\xi,\eta)$ defined as follows
\begin{equation*}
    \begin{aligned}
        r_1(\xi,\eta)=&
        [f(\xi,-\eta)+f(-\xi,\eta)]  / 2,\\
         r_2(\xi,\eta)=&
        [f(\xi,\eta)+f(-\xi,-\eta)]  / 2,\\
        j_1(\xi,\eta)=&
        [f(\xi,-\eta)-f(-\xi,\eta)]  / 2\varepsilon,\\
         j_2(\xi,\eta)=&
        [f(\xi,\eta)-f(-\xi,-\eta)]  / 2\varepsilon,\\
    \end{aligned}
\end{equation*}
thus we arrive at the system
\begin{equation}\label{whole system}
\left\{
   \begin{aligned}
       \varepsilon^2 \partial_t r_1+ \varepsilon^2\xi\partial_x j_1
       -\varepsilon^2\eta\partial_y j_1
       =&
       \rho-r_1,\\
       \varepsilon^2\partial_t r_2+ \varepsilon^2\xi\partial_x j_2
       +\varepsilon^2 \eta\partial_y j_2
       =&
       \rho-r_2,\\
       \varepsilon^2\partial_t j_1 
       +
       \xi\partial_x r_1
       -\eta\partial_y r_1
       =&
       -j_1,\\
       \varepsilon^2\partial_t j_2 +\xi\partial_x r_2
       +\eta\partial_y r_2
       =&
       -j_2.
   \end{aligned} 
   \right.
\end{equation}
In addition, note the fact:
\begin{equation}
    \rho=
    \frac{2}{\pi}\int_{\Omega^+}
    (\frac{r_1+r_2}{2})
    \,\mathrm{d}\boldsymbol{v}'
    :=\langle \frac{r_1+r_2}{2}\rangle,
\end{equation}
which satisfies the following equation:
\begin{equation}\label{Def: rho equation}
    2\partial_t \rho
    +
    \langle \xi\partial_x (j_1+j_2)+\eta\partial_y (j_2-j_1)\rangle
    =0.   
\end{equation}
Define $\frac{r_1+r_2}{2}=\rho+\varepsilon^2 w$, then we have
\begin{equation}
\begin{aligned}
    2\varepsilon^2\partial_t w
    +&
    \big(\xi\partial_x (j_1+j_2)+\eta\partial_y (j_2-j_1)\big)\\
    &-
    \langle \xi\partial_x (j_1+j_2)+\eta\partial_y (j_2-j_1)\rangle
    +
    2w=0.
    \end{aligned}
\end{equation}
Furthermore, denote $\varphi=r_2-r_1$, subtract $\eqref{whole system}_2$ from $\eqref{whole system}_1$ to deduce that
\begin{equation}
    \varepsilon^2\partial_t\varphi
    +
    \varepsilon^2\xi\partial_x(j_2-j_1)
    +\varepsilon^2\eta\partial_y(j_1+j_2)
    +\varphi
    =0.
\end{equation}
Adding and subtracting $\eqref{whole system}_3$ and $\eqref{whole system}_4$ yields
\begin{equation}
   \varepsilon^2\partial_t(j_1+j_2)+(j_1+j_2)
   +
   2\xi\partial_x(\rho+\varepsilon^2 w)+\eta\partial_y\varphi
   =0,
\end{equation}
and
\begin{equation}\label{Def: j1-j2 equation}
    \varepsilon^2\partial_t(j_2-j_1)+(j_2-j_1)
   +
   \xi\partial_x\varphi+2\eta\partial_y(\rho+\varepsilon^2w)
   =0.
\end{equation}
Thus, similarly to one dimension, we use the multiscale parity decomposition of the distribution function $f(t,\boldsymbol{x},\boldsymbol{v})$, explicitly separating the macroscopic density $\rho(t, \boldsymbol{x})$, $\varphi(t,\boldsymbol{x},\boldsymbol{v})$ at $\mathcal{O}(1)$, the odd term $j_1(t, \boldsymbol{x},\boldsymbol{v})+j_2(t, \boldsymbol{x},\boldsymbol{v})$, $j_2(t, \boldsymbol{x},\boldsymbol{v})-j_1(t, \boldsymbol{x},\boldsymbol{v})$  at $\mathcal{O}(\varepsilon)$, and higher-order corrections $w(t,\boldsymbol{x},\boldsymbol{v})$ at $\mathcal{O}(\varepsilon^2)$, each with definite parity properties, that is,
\begin{equation}\label{Def: scale decomposition-two}
\begin{aligned}
&f(t, \boldsymbol{x}, \boldsymbol{v}) 
=
\underbrace{\rho(t, \boldsymbol{x})+\frac{\varphi(t, \boldsymbol{x}, \boldsymbol{v})}{2}}_{\mathcal{O}(1)}\\
&+
\underbrace{\varepsilon \Big(\frac{j_2(t, \boldsymbol{x}, \boldsymbol{v})+j_1(t, \boldsymbol{x}, \boldsymbol{v})}{2}
+\frac{j_2(t, \boldsymbol{x}, \boldsymbol{v})-j_1(t, \boldsymbol{x}, \boldsymbol{v})}{2}\Big)}_{\mathcal{O}(\varepsilon)} 
+
\underbrace{\varepsilon^2 w(t, \boldsymbol{x}, \boldsymbol{v})}_{\mathcal{O}(\varepsilon^2)}.
\end{aligned}
\end{equation}
Hence, equations \eqref{Def: rho equation}-\eqref{Def: j1-j2 equation} constitute the multiscale parity decomposition framework of equation \eqref{linear transport equation2} to be used for the APNN:
\begin{equation}\label{Def: APNN system 2D}
\left\{
    \begin{aligned}
       & 2\partial_t \rho
    +
    \langle \xi\partial_x (j_1+j_2)+\eta\partial_y (j_2-j_1)\rangle
    =0\\
    &2\varepsilon^2\partial_t w
    +
    \big(\xi\partial_x (j_1+j_2)+\eta\partial_y (j_2-j_1)\big)\\
    &\quad-
    \langle \xi\partial_x (j_1+j_2)+\eta\partial_y (j_2-j_1)\rangle
    +
    2w=0\\
     &\varepsilon^2\partial_t\varphi
    +
    \varepsilon^2\xi\partial_x(j_2-j_1)+\varepsilon^2\eta\partial_y(j_1+j_2)
    +\varphi
    =0\\
     &\varepsilon^2\partial_t(j_1+j_2)+(j_1+j_2)
   +
   2\xi\partial_x(\rho+\varepsilon^2 w)+\eta\partial_y\varphi
   =0\\
   &\varepsilon^2\partial_t(j_2-j_1)+(j_2-j_1)
   +
   \xi\partial_x\varphi+2\eta\partial_y(\rho+\varepsilon^2 w)
   =0
    \end{aligned}
    \right.
\end{equation}
with the constraint $\langle w \rangle = 0$.

As $\varepsilon\rightarrow 0$, one has
\begin{equation}\label{limit1}
\varphi=0,\quad
j_1+j_2=-2\xi\partial_x \rho-\eta\partial_y\varphi,\quad
j_2-j_1=-2\eta\partial_y \rho-\xi\partial_x\varphi,
\end{equation}
and
\begin{equation}\label{limit2}
2w
=
    -\big(\xi\partial_x (j_1+j_2)+\eta\partial_y (j_2-j_1)\big)
    +
    \langle \xi\partial_x (j_1+j_2)+\eta\partial_y (j_2-j_1)\rangle.
\end{equation}
Applying \eqref{limit1} in $\eqref{Def: APNN system 2D}_1$  gives
\begin{equation*}
    \begin{aligned}
        \partial_t\rho
        -
        \langle\xi^2\partial_{xx}\rho+\eta^2\partial_{yy}\rho\rangle
        =0.
    \end{aligned}
\end{equation*}
Adding the above two equations and integrating over $\xi^2+\eta^2=1$, we obtain the diffusion equation
\begin{equation}\label{diffusion euqation2}
    \partial_t\rho
    =
    \frac{1}{2}(\partial_{xx}\rho+\partial_{yy}\rho).
\end{equation}
Furthermore, applying \eqref{limit1} and \eqref{diffusion euqation2} in \eqref{limit2}  yields
\begin{equation*}
    w
    =
    \big(\xi^2-\frac{1}{2}\big)\partial_{xx}\rho
    +
    \big(\eta^2-\frac{1}{2}\big)\partial_{yy}\rho.
\end{equation*}

\subsection{Loss function} 
We denote by $\rho_\theta$ and $\varphi_\theta$ the network approximations to the macroscopic variables $\rho$ and $\varphi$, and by $j_{1,\theta}$, $j_{2,\theta}$, $w_\theta$ the approximations to the microscopic fluxes $j_1$, $j_2$, and $w$, all corresponding to system \eqref{Def: APNN system 2D}. Here, $\theta$ collectively represents all trainable parameters. For the APNN method, the physics-informed loss is defined as the residual of the multiscale 
parity framework in system \eqref{Def: APNN system 2D}. Specifically, we denote
\begin{equation*}
\begin{aligned}
    d^{(1)}_\theta
    :=&
    \partial_t \rho_\theta
    +
    \langle \xi\partial_x (j_{1,\theta}+j_{2,\theta})+\eta\partial_y (j_{2,\theta}-j_{1,\theta})\rangle,\\
     d^{(2)}_\theta
    :=&
    2\varepsilon^2\partial_t w_\theta
    +
    \big(\xi\partial_x (j_{1,\theta}+j_{2,\theta})+\eta\partial_y (j_{2,\theta}-j_{1,\theta})\big)\\
    &-
    \langle \xi\partial_x (j_{1,\theta}+j_{2,\theta})+\eta\partial_y (j_{2,\theta}-j_{1,\theta})\rangle
    +
    2\omega_\theta,\\
    d^{(3)}_\theta
    :=&
   \partial_t\varphi_\theta
    +
   \xi\partial_x(j_{2,\theta}-j_{1,\theta})
   +\eta\partial_y(j_{1,\theta}+j_{2,\theta})
    +\frac{1}{\varepsilon^2}\varphi_\theta,\\
    d^{(4)}_\theta
    :=&
    \varepsilon^2\partial_t(j_{1,\theta}+j_{2,\theta})+(j_{1,\theta}+j_{2,\theta})
   +
   2\xi\partial_x(\rho_\theta+\varepsilon^2 w_\theta)+\eta\partial_y\varphi_\theta,\\
   d^{(5)}_\theta
    :=&
    \varepsilon^2\partial_t(j_{2,\theta}-j_{1,\theta})+(j_{2,\theta}-j_{1,\theta})
   +
   \xi\partial_x\varphi+2\eta\partial_y(\rho_\theta+\varepsilon^2 w_\theta).
    \end{aligned}
\end{equation*}
The corresponding loss function is then defined as:
\begin{equation*}
    \mathcal{R}^\varepsilon_{\mathrm{APNN}}
    =
    \mathcal{R}_{\mathrm{residual}}^\varepsilon
    +
    \mathcal{R}_{\mathrm{initial}}^\varepsilon
    +
     \mathcal{R}_{\mathrm{boundary}}^\varepsilon,
\end{equation*}
where the terms on the right-hand side of the above equation are defined as
\begin{equation*}
    \begin{aligned}
         \mathcal{R}_{\mathrm{residual}}^\varepsilon
         =&
         \frac{1}{|\mathcal{T}\times\mathcal{D}|}\int_{\mathcal{T}}\int_{\mathcal{D}}
         |d^{(1)}_\theta|^2\,\mathrm{d}\boldsymbol{x}\mathrm{d}t
         +
         \frac{1}{|\mathcal{T}\times\mathcal{D}\times \Omega^+|}\sum_{j=2}^5\int_\mathcal{T}\int_{\mathcal{D}}\int_{\Omega^+}
         |d^{(j)}_\theta|^2\,
         \mathrm{d}\boldsymbol{v}\mathrm{d}\boldsymbol{x}\mathrm{d}t,
    \end{aligned}
\end{equation*}
\begin{equation*}
    \begin{aligned}
        \mathcal{R}^\varepsilon_{\mathrm{initial}}
        &=
         \frac{\lambda_1}{|\mathcal{D}|}\int_{\mathcal{D}}|\rho_\theta-\rho_{\mathrm{IC}}|^2\,\mathrm{d}\boldsymbol{x}
         +
        \frac{\lambda_2}{|\mathcal{D}\times \Omega^+|}\int_\mathcal{D}\int_{\Omega^+}
        \varepsilon^4|w_\theta-w_{\mathrm{IC}}|^2
        +
        |\varphi_\theta-\varphi_\mathrm{IC}|^2\\
       & +
        \varepsilon^2\big(|(j_{1,\theta}+j_{2,\theta})-(j_{1,\mathrm{IC}}+j_{2,\mathrm{IC}})|^2
        +
        |(j_{2,\theta}-j_{1,\theta})-(j_{2,\mathrm{IC}}-j_{1,\mathrm{IC}})|^2\big)\,
        \mathrm{d}\boldsymbol{v}\mathrm{d}\boldsymbol{x},
    \end{aligned}
\end{equation*}
and
\begin{equation*}
    \begin{aligned}
        \mathcal{R}^\varepsilon_{\mathrm{boundary}}
        =&
        \frac{\lambda_3}{|\mathcal{T}\times\partial\mathcal{D}|}\int_{\mathcal{T}}
         \int_{\mathcal{D}}|\rho_\theta-\rho_{\mathrm{BC}}|^2\,\mathrm{d}s\mathrm{d}t\\
         &+
        \frac{\lambda_4}{|\mathcal{T}\times\partial\mathcal{D}\times \Omega^+|}\int_{\mathcal{T}}\int_\mathcal{D}\int_{\Omega^+}
        |w_\theta-w_{\mathrm{BC}}|^2
        +
        |\varphi_\theta-\varphi_\mathrm{BC}|^2
        +
        |(j_{1,\theta}+j_{2,\theta})\\
        &-
        (j_{1,\mathrm{BC}}+j_{2,\mathrm{BC}})|^2
        +
        |(j_{2,\theta}-j_{1,\theta})
        -(j_{2,\mathrm{BC}}-j_{1,\mathrm{BC}})|^2\,
        \mathrm{d}\boldsymbol{v}\mathrm{d}s\mathrm{d}t. 
    \end{aligned}
\end{equation*}
Here, $\lambda_i\,(i=1,\cdots,4)$ are hyperparameters as penalty terms. As $\varepsilon\rightarrow 0$, $\mathcal{R}^\varepsilon_{\mathrm{residual}}$ converges to 
\begin{equation*}
    \begin{aligned}
        \mathcal{R}_{\mathrm{residual}}^0
        :=&
        \frac{1}{|\mathcal{T}\times\mathcal{D}|}\int_{\mathcal{T}}\int_{\mathcal{D}}
         |d^{(1)}_\theta|^2\,\mathrm{d}\boldsymbol{x}\mathrm{d}t
         +
         \frac{1}{|\mathcal{T}\times\mathcal{D}\times \Omega^+|}
         \int_\mathcal{T}\int_{\mathcal{D}}\int_{\Omega^+}
         |\varphi_\theta|^2\,\mathrm{d}\boldsymbol{v}\mathrm{d}\boldsymbol{x}\mathrm{d}t\\
        & +
         \frac{1}{|\mathcal{T}\times\mathcal{D}\times \Omega^+|}\sum_{j=2,4,5}\int_\mathcal{T}\int_{\mathcal{D}}\int_{\Omega^+}|\lim_{\varepsilon\rightarrow 0}d^{(j)}_\theta|^2\,
         \mathrm{d}\boldsymbol{v}\mathrm{d}\boldsymbol{x}\mathrm{d}t,
    \end{aligned}
\end{equation*}
which is exactly the loss for the limit equations \eqref{Def: rho equation}--\eqref{Def: j1-j2 equation}. Here
\begin{equation*}
    \begin{aligned}
        \lim_{\varepsilon\rightarrow 0}d^{(2)}_\theta
        =&
        \big(\xi\partial_x (j_{1,\theta}+j_{2,\theta})+\eta\partial_y (j_{2,\theta}-j_{1,\theta})\big)\\
    &-
    \langle \xi\partial_x (j_{1,\theta}+j_{2,\theta})+\eta\partial_y (j_{2,\theta}-j_{1,\theta})\rangle
    +
    2\omega_\theta,\\
    \lim_{\varepsilon\rightarrow 0}d^{(4)}_\theta
    =&
    (j_{1,\theta}+j_{2,\theta})
   +
   2\xi\partial_x\rho_\theta+\eta\partial_y\varphi_\theta,\\
   \lim_{\varepsilon\rightarrow 0}d^{(5)}_\theta
   =&
   (j_{2,\theta}-j_{1,\theta})
   +
   \xi\partial_x\varphi_{\theta}+2\eta\partial_y\rho_\theta.
    \end{aligned}
\end{equation*}

Analogous to the one-dimensional case, the approximation property likewise ensures the convergence of the APNN loss function in two dimensions. 
\begin{theorem}
    Let $\rho\in H^2(\mathcal{T}\times\mathcal{D}),\ \varphi,j_1,j_2,w\in H^2(\mathcal{T}\times\mathcal{D}\times\Omega)$ be the analytic solution to the multiscale parity decomposition system \eqref{Def: APNN system 2D}, and $N>5$ be an integer. Then there exist APNN solutions $\hat{\rho}^N,\hat{\varphi}^N,\hat{j}_1^N,\hat{j}_2^N,\hat{w}^N$ such that
   \begin{equation*}
        \mathcal{R}^\varepsilon_{\mathrm{APNN}}\rightarrow 0,\quad \mbox{as } N\rightarrow \infty.
    \end{equation*}
\end{theorem}

The proof proceeds along similar lines to that of Theorem \ref{Thm: convergence of the loss fuction}, and is therefore omitted for brevity.

\subsection{Uniform error method}
In the following, we demonstrate that the APNN solution converges uniformly to the analytical solution of the two-dimensional RTE \eqref{linear transport equation2}.
Parallel to our approach in one dimension, the total error $\mathcal{R}_{\mathrm{total}}^\varepsilon$ for the 2D system \eqref{Def: APNN system 2D} is defined as the difference between the analytic and APNN solutions:
    \begin{equation}\label{Def: total error 2D}
    \begin{aligned}
        \mathcal{R}_{\mathrm{total}}^\varepsilon
        =&
        \|f-f_\theta\|_{L^2(\mathcal{T}\times\mathcal{D}\times\Omega^+)}^2.
        \end{aligned}
    \end{equation}
 According to the decomposition \eqref{Def: scale decomposition-two}, it yields that
\begin{equation*}
\begin{aligned}
    \mathcal{R}_{\mathrm{total}}^\varepsilon
    \leq&
    \|\rho-\rho_\theta\|_{L^2(\mathcal{D})}^2
    +
    \|\varphi-\varphi_\theta\|_{L^2(\mathcal{D}\times \Omega^+)}^2
    +
    \varepsilon^4\|w-w_\theta\|_{L^2(\mathcal{D}\times \Omega^+)}^2\\
    +&
    \varepsilon^2\|(j_1+j_2)-(j_{1,\theta}+j_{2,\theta})\|_{L^2(\mathcal{D}\times \Omega^+)}^2
    +
    \varepsilon^2\|(j_2-j_1)-(j_{2,\theta}-j_{1,\theta})\|_{L^2(\mathcal{D}\times \Omega^+)}^2.
    \end{aligned}
\end{equation*}

Through a direct approach, a bound on the total error is established in the following theorem.
\begin{theorem}\label{Thm: convergence of the APNN solution-classical 2D}
    Let $\rho\in H^{1}(\mathcal{T}\times\mathcal{D})$, $\varphi,j_1,j_2,w\in H^{1}(\mathcal{T}\times\mathcal{D}\times \Omega^+)$ be the analytic solution to the multiscale parity decomposition system \eqref{Def: APNN system 2D}. Then
    \begin{equation*}
        \mathcal{R}_{\mathrm{total}}^\varepsilon
        \leq
        \frac{C}{\varepsilon}\mathcal{R}^\varepsilon_\mathrm{APNN},
    \end{equation*}
    where the constant $C>0$ is independent of $\varepsilon$.
\end{theorem}
\begin{proof}
Denote $\tilde{\rho}=\rho_\theta-\rho$, $\tilde{w}=w_\theta-w$, $\tilde{\varphi}=\varphi_\theta-\varphi$, and $\tilde{j}_i=j_{i,\theta}-j_i$ for $i=1,2$.
    Since $\rho$ satisfies \eqref{Def: rho equation}, we have
    \begin{equation}\label{tilde rho equation}
        2\partial_t \tilde{\rho}
    +
    \langle \xi\partial_x (\tilde{j}_1+\tilde{j}_2)+\eta\partial_y (\tilde{j}_2-\tilde{j}_1)\rangle
    =d^{(1)}_\theta.
    \end{equation}
    Similarly, it yields that
    \begin{equation}\label{tilde w equation}
        \begin{aligned}
            2\varepsilon^2\partial_t \tilde{w}
    +&
    \big(\xi\partial_x (\tilde{j}_1+\tilde{j}_2)+\eta\partial_y (\tilde{j}_2-\tilde{j}_1)\big)\\
    &-
    \langle \xi\partial_x (\tilde{j}_1+\tilde{j}_2)+\eta\partial_y (\tilde{j}_2-\tilde{j}_1)\rangle
    +
    2\tilde{w}=d^{(2)}_\theta,
        \end{aligned}
    \end{equation}
    \begin{equation}\label{tilde r1-r2 equation}
   \partial_t\tilde{\varphi}
    +
    \xi\partial_x(\tilde{j}_2-\tilde{j}_1)+\eta\partial_y(\tilde{j}_1+\tilde{j}_2)
    +
    \frac{1}{\varepsilon^2}\tilde{\varphi}
    =d^{(3)}_\theta,
\end{equation}
\begin{equation}\label{tilde j1+j2 equation}
   \varepsilon^2\partial_t(\tilde{j}_1+\tilde{j}_2)+(\tilde{j}_1+\tilde{j}_2)
   +
   2\xi\partial_x(\tilde{\rho}+\varepsilon^2\tilde{w})+\eta\partial_y\tilde{\varphi}
   =d^{(4)}_\theta,
\end{equation}
and
\begin{equation}\label{tilde j1-j2 equation}
    \varepsilon^2\partial_t(\tilde{j}_2-\tilde{j}_1)+(\tilde{j}_2-\tilde{j}_1)
   +
   \xi\partial_x\tilde{\varphi}+2\eta\partial_y(\tilde{\rho}+\varepsilon^2\tilde{w})
   =d^{(5)}_\theta.
\end{equation}
 Recalling the decomposition of $f$ in \eqref{Def: scale decomposition-two}, we use the notation $\tilde{f}=f-f_\theta$ with 
 \begin{equation*}
\begin{aligned}
&\tilde{f} 
=
\tilde{\rho} +\frac{\tilde{\varphi} }{2}
+
\varepsilon \Big(\frac{\tilde{j}_2+\tilde{j}_1}{2}
+\frac{\tilde{j}_2-\tilde{j}_1}{2}\Big) 
+
\varepsilon^2 \tilde{w}.
\end{aligned}
\end{equation*}
Hence, we deduce that
\begin{equation*}
    \begin{aligned}
     \varepsilon\partial_t\tilde{f}
        +\xi\partial_x\tilde{f}+\eta\partial_y\tilde{f}
        =
        \frac{1}{\varepsilon}(\tilde{\rho}-\tilde{f})
        +
        \frac{1}{2}\Big(\varepsilon\sum_{i=1}^3d^{(i)}_\theta+d^{(4)}_\theta+d^{(5)}_\theta\Big).
    \end{aligned}
\end{equation*}
Observing that its structure is analogous to the one presented in the proof of Theorem \ref{Thm: convergence of the APNN solution-classical}, the desired estimate follows directly by analogous reasoning.

\end{proof}

Building on this multiscale parity decomposition framework \eqref{Def: APNN system 2D}, we are motivated to establish an improved estimate for the error.
\begin{theorem}\label{Thm: convergence of the APNN solution-2D}
    Let $\rho\in H^{1}(\mathcal{T}\times\mathcal{D})$, $\varphi,j_1,j_2,w\in H^{1}(\mathcal{T}\times\mathcal{D}\times \Omega^+)$ be the analytic solution to the multiscale parity decomposition system \eqref{Def: APNN system 2D}. Then
    \begin{equation*}
        \mathcal{R}_{\mathrm{total}}^\varepsilon
        \leq
        C(\mathcal{R}^\varepsilon_\mathrm{APNN}+\varepsilon^2),
    \end{equation*}
    where the constant $C>0$ is independent of $\varepsilon$.
\end{theorem}

\begin{proof}
Following Theorem \ref{Thm: convergence of the APNN solution-classical 2D}, we continue to use the notations established in \eqref{tilde rho equation}-\eqref{tilde j1+j2 equation}.
Multiply \eqref{tilde rho equation} by $\tilde{\rho}$ and integrate over $\mathcal{D}$ to deduce that
\begin{equation}\label{proof of Thm: convergence of the APNN solution-2D rho}
    \frac{\mathrm{d}}{\mathrm{d}t}\|\tilde{\rho}\|_{L^2(\mathcal{D})}^2
    +
    \frac{2}{\pi}\int_{\mathcal{D}}\int_{\Omega^+}\big(\xi\partial_x (\tilde{j}_1+\tilde{j}_2)+\eta\partial_y (\tilde{j}_2-\tilde{j}_1)\big)\tilde{\rho}\,
    \mathrm{d}\boldsymbol{v}\mathrm{d}\boldsymbol{x}
    =
    \int_{\mathcal{D}}d^{(1)}_\theta\tilde{\rho}\,\mathrm{d}\boldsymbol{x}.
\end{equation}
We multiply \eqref{tilde w equation}, \eqref{tilde r1-r2 equation}, \eqref{tilde j1+j2 equation}, \eqref{tilde j1-j2 equation}  by $\varepsilon^2\tilde{w}$, $\tilde{\varphi}$, $\tilde{j}_1+\tilde{j}_2$, and $\tilde{j}_2-\tilde{j}_1$, respectively, and integrate over $\mathcal{D}\times \Omega^+$, we get
\begin{equation}\label{proof of Thm: convergence of the APNN solution-2D w}
    \begin{aligned}
        \varepsilon^4&\frac{\mathrm{d}}{\mathrm{d}t}\|\tilde{w}\|_{L^2(\mathcal{D}\times \Omega^+)}^2
        +
        2\varepsilon^2\|\tilde{w}\|_{L^2(\mathcal{D}\times \Omega^+)}^2\\
       & +
\varepsilon^2\int_{\mathcal{D}}\int_{\Omega^+}
\big(\xi\partial_x (\tilde{j}_1+\tilde{j}_2)+\eta\partial_y (\tilde{j}_2-\tilde{j}_1)\big)\tilde{w}\,
        \mathrm{d}\boldsymbol{v}\mathrm{d}\boldsymbol{x}\\
        &=
        \varepsilon^2\int_{\mathcal{D}}\int_{\Omega^+}
        d^{(2)}_\theta\tilde{w}\,
        \mathrm{d}\boldsymbol{v}\mathrm{d}\boldsymbol{x}
        +
        \varepsilon^2\int_{\mathcal{D}}\int_{\Omega^+}
        \langle \xi\partial_x (\tilde{j}_1+\tilde{j}_2)+\eta\partial_y (\tilde{j}_2-\tilde{j}_1)\rangle\tilde{w}\,
        \mathrm{d}\boldsymbol{v}\mathrm{d}\boldsymbol{x},
    \end{aligned}
\end{equation}
\begin{equation}\label{proof of Thm: convergence of the APNN solution-2D varphi}
    \begin{aligned}
        &\frac{1}{2}\frac{\mathrm{d}}{\mathrm{d}t}\|\tilde{\varphi}\|_{L^2(\mathcal{D}\times \Omega^+)}^2
        +
        \frac{1}{\varepsilon^2}\|\tilde{\varphi}\|_{L^2(\mathcal{D}\times \Omega^+)}^2\\
        &+
        \int_{\mathcal{D}}\int_{\Omega^+}
        \big(\xi\partial_x(\tilde{j}_2-\tilde{j}_1)+\eta\partial_y(\tilde{j}_1+\tilde{j}_2)\big)\tilde{\varphi}\,
        \mathrm{d}\boldsymbol{v}\mathrm{d}\boldsymbol{x}
        =
        \int_{\mathcal{D}}\int_{\Omega^+}
        d^{(3)}_\theta\tilde{\varphi}\,
        \mathrm{d}\boldsymbol{v}\mathrm{d}\boldsymbol{x},
    \end{aligned}
\end{equation}
\begin{equation}\label{proof of Thm: convergence of the APNN solution-2D j1+j2}
    \begin{aligned}
        \frac{\varepsilon^2}{2}\frac{\mathrm{d}}{\mathrm{d}t}
        \|\tilde{j}_1+\tilde{j}_2\|_{L^2(\mathcal{D}\times \Omega^+)}^2
        &+
        \int_{\mathcal{D}}\int_{\Omega^+}
\big(2\xi\partial_x(\tilde{\rho}+\varepsilon^2\tilde{w})+\eta\partial_y\tilde{\varphi}\big)(\tilde{j}_1+\tilde{j}_2)\,
        \mathrm{d}\boldsymbol{v}\mathrm{d}\boldsymbol{x}\\
        +
        \|\tilde{j}_1+\tilde{j}_2\|_{L^2(\mathcal{D}\times \Omega^+)}^2
       & =
        \int_{\mathcal{D}}\int_{\Omega^+}
        d^{(4)}_\theta(\tilde{j}_1+\tilde{j}_2)\,
        \mathrm{d}\boldsymbol{v}\mathrm{d}\boldsymbol{x},
    \end{aligned}
\end{equation}
and
\begin{equation}\label{proof of Thm: convergence of the APNN solution-2D j1-j2}
    \begin{aligned}
        \frac{\varepsilon^2}{2}\frac{\mathrm{d}}{\mathrm{d}t}\|\tilde{j}_2-\tilde{j}_1\|_{L^2(\mathcal{D}\times \Omega^+)}^2
        &+
        \int_{\mathcal{D}}\int_{\Omega^+}
\big(\xi\partial_x\tilde{\varphi}+2\eta\partial_y(\tilde{\rho}+\varepsilon^2\tilde{w})\big)(\tilde{j}_2-\tilde{j}_1)\,
        \mathrm{d}\boldsymbol{v}\mathrm{d}\boldsymbol{x}\\
        +
         \|\tilde{j}_2-\tilde{j}_1\|_{L^2(\mathcal{D}\times \Omega^+)}^2
        & =
        \int_{\mathcal{D}}\int_{\Omega^+}
        d^{(5)}_\theta(\tilde{j}_2-\tilde{j}_1)\,
        \mathrm{d}\boldsymbol{v}\mathrm{d}\boldsymbol{x},
    \end{aligned}
\end{equation}
We begin by estimating the cross-terms on the left-hand sides of \eqref{proof of Thm: convergence of the APNN solution-2D rho}--\eqref{proof of Thm: convergence of the APNN solution-2D j1-j2}. Define
\begin{equation*}
    \begin{aligned}
    I_1
:=&\
2\big(\xi\partial_x (\tilde{j}_1+\tilde{j}_2)+\eta\partial_y (\tilde{j}_2-\tilde{j}_1)\big)\tilde{\rho}
        +
        2\varepsilon^2\big(\xi\partial_x (\tilde{j}_1+\tilde{j}_2)+\eta\partial_y (\tilde{j}_2-\tilde{j}_1)\big)\tilde{w}\\
        &+
        \big(\xi\partial_x(\tilde{j}_2-\tilde{j}_1)+\eta\partial_y(\tilde{j}_1+\tilde{j}_2)\big)\tilde{\varphi},\\
         I_2
        :=&\
        \big(2\xi\partial_x(\tilde{\rho}+\varepsilon^2\tilde{w})+\eta\partial_y\tilde{\varphi}\big)(\tilde{j}_1+\tilde{j}_2)
    +
    \big(\xi\partial_x\tilde{\varphi}+2\eta\partial_y(\tilde{\rho}+\varepsilon^2\tilde{w})\big)(\tilde{j}_2-\tilde{j}_1).
    \end{aligned}
\end{equation*}
Then a direct computation shows
\begin{equation*}
    \begin{aligned}
      I_1+I_2 
      =&
      \xi\partial_x\big((\tilde{j}_2-\tilde{j}_1)\tilde{\varphi}\big)
      +
      \eta\partial_y\big((\tilde{j}_1+\tilde{j}_2)\tilde{\varphi}\big)\\
      &+
       2\xi\partial_x\big((\tilde{\rho}+\varepsilon^2\tilde{w})(\tilde{j}_1+\tilde{j}_2)\big)
      +
      2\eta\partial_y\big((\tilde{\rho}+\varepsilon^2\tilde{w})(\tilde{j}_2-\tilde{j}_1)\big),
    \end{aligned}
\end{equation*}
and thus, by the divergence theorem and the periodic boundary condition, we have
\begin{equation}
\begin{aligned}
    \int_{\mathcal{D}}&\int_{\Omega^+}
    I_1+I_2\,
    \mathrm{d}\boldsymbol{v}\mathrm{d}\boldsymbol{x}=0.
    \end{aligned}
\end{equation}
 Furthermore, by applying Young's inequality to the terms on the right-hand side of \eqref{proof of Thm: convergence of the APNN solution-2D rho}--\eqref{proof of Thm: convergence of the APNN solution-2D varphi}, we obtain
\begin{equation*}
    \begin{aligned}
        \pi\int_{\mathcal{D}}d^{(1)}_\theta\tilde{\rho}\,\mathrm{d}\boldsymbol{x}
        &+
        2\varepsilon^2\int_{\mathcal{D}}\int_{\Omega^+}
        d^{(2)}_\theta\tilde{w}\,
        \mathrm{d}\boldsymbol{v}\mathrm{d}\boldsymbol{x}
        +
        \int_{\mathcal{D}}\int_{\Omega^+}
        d^{(3)}_\theta\tilde{\varphi}\,
        \mathrm{d}\boldsymbol{v}\mathrm{d}\boldsymbol{x}\\
        \leq&
        C\Big(\|d^{(1)}_\theta\|_{L^2(\mathcal{D})}^2
        +
        \|d^{(2)}_\theta\|_{L^2(\mathcal{D}\times \Omega^+)}^2
        +
        \|d^{(3)}_\theta\|_{L^2(\mathcal{D}\times \Omega^+)}^2\Big)\\
       & +
        \Big(\pi\|\tilde{\rho}\|_{L^2(\mathcal{D})}^2
        +
        \varepsilon^4\|\tilde{w}\|_{L^2(\mathcal{D}\times \Omega^+)}^2
        +
        \|\tilde{\varphi}\|_{L^2(\mathcal{D}\times \Omega^+)}^2\Big),
    \end{aligned}
\end{equation*}
and
\begin{equation*}
\begin{aligned}
    2\varepsilon^2\int_{\mathcal{D}}\int_{\Omega^+}&
        \langle \xi\partial_x (\tilde{j}_1+\tilde{j}_2)+\eta\partial_y (\tilde{j}_2-\tilde{j}_1)\rangle\tilde{w}\,
        \mathrm{d}\boldsymbol{v}\mathrm{d}\boldsymbol{x}\\
        \leq&
        C\varepsilon^2\Big(\|\tilde{j}_1\|_{L^2(\mathcal{D}\times \Omega^+)}^2
        +
        \|\tilde{j}_2\|_{L^2(\mathcal{D}\times \Omega^+)}^2\Big)
        +
        \varepsilon^2\|\tilde{w}\|_{L^2(\mathcal{D}\times \Omega^+)}^2. 
        \end{aligned}
\end{equation*}
Similarly, for the right-hand side terms of \eqref{proof of Thm: convergence of the APNN solution-2D j1+j2}, \eqref{proof of Thm: convergence of the APNN solution-2D j1-j2}, we have
\begin{equation*}
\begin{aligned}
     \int_{\mathcal{D}}\int_{\Omega^+}
        d^{(4)}_\theta(\tilde{j}_1&+\tilde{j}_2)\,
        \mathrm{d}\boldsymbol{v}\mathrm{d}\boldsymbol{x}
        +
        \int_{\mathcal{D}}\int_{\Omega^+}
        d^{(5)}_\theta(\tilde{j}_2-\tilde{j}_1)\,
        \mathrm{d}\boldsymbol{v}\mathrm{d}\boldsymbol{x}
        \leq
        \frac{1}{2}
        \Big(\|d^{(4)}_\theta\|_{L^2(\mathcal{D}\times \Omega^+)}^2\\
        &+
        \|d^{(5)}_\theta\|_{L^2(\mathcal{D}\times \Omega^+)}^2
        +
        \|\tilde{j}_1+\tilde{j}_2\|_{L^2(\mathcal{D}\times \Omega^+)}^2
        +
        \|\tilde{j}_2-\tilde{j}_1\|_{L^2(\mathcal{D}\times \Omega^+)}^2\Big).
        \end{aligned}
\end{equation*}
Hence, we conclude that
\begin{equation*}
    \begin{aligned}
        \frac{\mathrm{d}}{\mathrm{d}t}E(t)
        \leq&
        C\Big(\|d^{(1)}_\theta\|_{L^2(\mathcal{D})}^2
        +
        \sum_{j=2}^5\|d^{(j)}_\theta\|_{L^2(\mathcal{D}\times \Omega^+)}^2\\
        &+
        \varepsilon^2\big(\|\tilde{j}_1\|_{L^2(\mathcal{D}\times \Omega^+)}^2
        +
        \|\tilde{j}_2\|_{L^2(\mathcal{D}\times \Omega^+)}^2\big)
        \Big)
        +
        E(t),
    \end{aligned}
\end{equation*}
where 
\begin{equation*}
\begin{aligned}
    E(t)
    =&
    \pi\|\tilde{\rho}\|_{L^2(\mathcal{D})}^2
        +
        \varepsilon^4\|\tilde{w}\|_{L^2(\mathcal{D}\times \Omega^+)}^2\\
        &+
        \|\tilde{\varphi}\|_{L^2(\mathcal{D}\times \Omega^+)}^2
        +
        \varepsilon^2\|\tilde{j}_1+\tilde{j}_2\|_{L^2(\mathcal{D}\times \Omega^+)}^2
        +
        \varepsilon^2\|\tilde{j}_2-\tilde{j}_1\|_{L^2(\mathcal{D}\times \Omega^+)}^2.
        \end{aligned}
\end{equation*}
Therefore, an application of Gronwall's inequality concludes the proof.
 
\end{proof}

\section{Structure-preserving mechanism}\label{Sec: Structure-preserving mechanism}
The structure-preserving mechanism plays a crucial role in ensuring that the neural approximation retains the intrinsic symmetries and conservation properties of the AP system. 
In the context of neural networks, it is particularly important to guarantee that the learned representations are consistent with the analytical properties of the underlying kinetic model, especially in the diffusive limit $\varepsilon \to 0$.

\subsection{One-dimensional parity symmetry, conservation and positivity}
For the AP system~\eqref{Def: APNN system}, the variables $(\rho, j, w)$ correspond to the macroscopic density, the odd flux, and the higher-order even correction. 
The neural network approximations $(\rho_\theta, j_\theta, w_\theta)$ are constructed to respect the following relations
\begin{equation}
j_\theta(t,x,-v) = -\,j_\theta(t,x,v), 
w_\theta(t,x,-v) = w_\theta(t,x,v),
\langle w_\theta(t,x,\cdot) \rangle = 0,
\end{equation}
which are essential for preserving the equivalence between kinetic and macroscopic moments. These relations are explicitly embedded in neural architectures as follows:
\begin{itemize}
\item The flux network $j_\theta$ is enforced to be odd in $v$ by the antisymmetric combination
\begin{equation*}
    j_\theta(t,x,v) = \tfrac{1}{2}\big[ j_\text{net}(t,x,v) - j_\text{net}(t,x,-v) \big],
\end{equation*}
ensuring that the velocity reversal symmetry is exactly preserved at the network level.

\item The correction term $w_\theta$ is constructed to be even in $v$ and to satisfy the zero-average constraint by
\begin{equation*}
    w_\theta(t,x,v) = \tfrac{1}{2}\Big[
    w_\text{net}(t,x,v) - \langle w_\text{net}(t,x,v) \rangle + w_\text{net}(t,x,-v) - \langle w_\text{net}(t,x,-v) \rangle
\Big],
\end{equation*}
which guarantees $\langle w_\theta \rangle = 0$ for all $(t,x)$.

\item The macroscopic density $\rho_\theta$ is constrained to remain positive through the transformation $\rho_\theta = \log(1 + \exp(\rho_{\text{net}}))$.
This use of the softplus activation function naturally bounds the gradient and mitigates numerical instabilities.
\end{itemize}
Here, $j_\text{net}, w_\text{net}, \rho_{\text{net}}$ denote the raw outputs of the neural submodules before enforcing these structures. 

\subsection{Two-dimensional parity symmetry, conservation and positivity}
For the AP system \eqref{Def: APNN system 2D} with velocity variable $(\xi,\eta)$ on the unit circle satisfying $\xi^2+\eta^2=1$, the structure-preserving mechanism must account for the more complex parity structure introduced by the decomposition with respect to the vector
$(\xi, \eta)$. Recall that \eqref{Def: APNN system 2D} involves five variables: the macroscopic density $\rho$, the variable $\varphi$, two odd fluxes $j_1, j_2$, and the higher-order even correction $w$.
The neural network approximations are constructed to satisfy the following symmetry and conservation properties:
\begin{equation}
\begin{aligned}
    & \varphi_\theta(t, \boldsymbol{x}, \xi, \eta) = \varphi_\theta(t, \boldsymbol{x}, -\xi, -\eta) = -\varphi_\theta(t,\boldsymbol{x},-\xi,\eta) = -\varphi_\theta(t,\boldsymbol{x},\xi,-\eta), \\
    & j_{1,\theta}(t, \boldsymbol{x}, \xi, \eta) = -j_{1,\theta}(t, \boldsymbol{x}, -\xi, -\eta), \, j_{2,\theta}(t, \boldsymbol{x}, \xi, \eta) = -j_{2,\theta}(t, \boldsymbol{x}, -\xi, -\eta), \\
    & w_\theta(t, \boldsymbol{x}, \xi, \eta) = w_\theta(t, \boldsymbol{x}, -\xi, -\eta) = w_\theta(t, \boldsymbol{x}, -\xi, \eta) = w_\theta(t, \boldsymbol{x}, \xi, -\eta), \\
    & \langle w_\theta(t, \boldsymbol{x}, \xi, \eta) \rangle = 0 .
\end{aligned}
\end{equation}
These properties are enforced in the neural architecture as follows:
\begin{itemize}
    \item The anisotropy term $\varphi_\theta$ is constructed to satisfy the required parity structure by combining evaluations at all four symmetric velocity points:
    \begin{equation*}\label{eq:phi_structure}
    \begin{aligned}
        \varphi_\theta(t,\boldsymbol{x},\xi,\eta) = \tfrac{1}{2}\Big[&\varphi_{\text{net}}(t,\boldsymbol{x},\xi,\eta) + \varphi_{\text{net}}(t,\boldsymbol{x},-\xi,-\eta) \\
        &- \varphi_{\text{net}}(t,\boldsymbol{x},-\xi,\eta) - \varphi_{\text{net}}(t,\boldsymbol{x},\xi,-\eta)\Big],
    \end{aligned}
    \end{equation*}
    ensuring $\varphi_\theta(t,\boldsymbol{x},\xi,\eta) = \varphi_\theta(t,\boldsymbol{x},-\xi,-\eta)$ and $\varphi_\theta(t,\boldsymbol{x},\xi,\eta) = -\varphi_\theta(t,\boldsymbol{x},-\xi,\eta)$. 
    
    \item The fluxes $j_{1,\theta}$ and $j_{2,\theta}$ are enforced to be odd with respect to the transformation $(\xi,\eta) \to (-\xi,\eta)$ by
    \begin{equation*}\label{eq:j_structure}
    \begin{aligned}
        j_{1,\theta}(t,\boldsymbol{x},\xi,\eta) &= \tfrac{1}{2}\left[j_{1,\text{net}}(t,\boldsymbol{x},\xi,-\eta) - j_{1,\text{net}}(t,\boldsymbol{x},-\xi,\eta)\right], \\
        j_{2,\theta}(t,\boldsymbol{x},\xi,\eta) &= \tfrac{1}{2}\left[j_{2,\text{net}}(t,\boldsymbol{x},\xi,\eta) - j_{2,\text{net}}(t,\boldsymbol{x},-\xi,-\eta)\right],
    \end{aligned}
    \end{equation*}
    which are consistent with the definitions from the previous decomposition. 
    
    \item The higher-order correction $w_\theta$ is constructed to be even with respect to all velocity reflections and to satisfy the zero-average constraint. This is achieved through the symmetrized construction:
    \begin{equation*}\label{eq:w_structure_2d}
    \begin{aligned}
        w_\theta(t,\boldsymbol{x},\xi,\eta) = \tfrac{1}{2}\Big[&w_{\text{net}}(t,\boldsymbol{x},\xi,\eta) - \langle w_{\text{net}}(t,\boldsymbol{x},\xi,\eta)\rangle \\
        &+ w_{\text{net}}(t,\boldsymbol{x},-\xi,-\eta) - \langle w_{\text{net}}(t,\boldsymbol{x},-\xi,-\eta)\rangle \\
        &+ w_{\text{net}}(t,\boldsymbol{x},-\xi,\eta) - \langle w_{\text{net}}(t,\boldsymbol{x},-\xi,\eta)\rangle \\
        &+ w_{\text{net}}(t,\boldsymbol{x},\xi,-\eta) - \langle w_{\text{net}}(t,\boldsymbol{x},\xi,-\eta)\rangle\Big],
    \end{aligned}
    \end{equation*}
    guaranteeing both the even parity with respect to all coordinate reflections and the constraint $\langle w_\theta \rangle = 0$. The zero-average property is essential for maintaining consistency with the decomposition $\frac{r_1+r_2}{2} = \rho + \varepsilon^2 w$.
    
    \item As in the one-dimensional case, the macroscopic density $\rho_\theta$ is enforced to remain positive through the transformation $\rho_\theta = \log(1 + \exp(\rho_{\text{net}}))$, ensuring physical consistency by construction.
\end{itemize}
Here, $\varphi_{\text{net}}$, $j_{1,\text{net}}$, $j_{2,\text{net}}$, $w_{\text{net}}$, and $\rho_{\text{net}}$ denote the raw outputs of the corresponding neural submodules before enforcing the structural constraints. 

By integrating these constraints directly into the neural representation, our framework not only reproduces the correct asymptotic behavior but also retains the fundamental structure of the parity equations, thereby ensuring both accuracy and physical consistency across multiple scales.

\section{Numerical experiments}\label{Sec: Numerical experiments} 
In this section, we present a series of numerical experiments performed for both the rarefied and diffusive regimes of one- and two-dimensional RTEs, to demonstrate the effectiveness of the proposed framework and validate the theoretical analysis.

For these experiments, we employ a four-block AdaptiveResNet architecture, using $128$ and $256$ hidden units for the one- and two-dimensional problems, respectively. An $L$ layer AdaptiveResNet is proposed as follows:
\begin{equation}
    \begin{aligned}
        g_{\theta}^{[0]}(z) & = W^{[0]} z + b^{[0]},   \\
        g_{\theta}^{[l]}(z) & = \text{AdaptiveResidualBlock}^{[l]}(g_{\theta}^{[l-1]}(z)), \, 1 \le l \le L-1, \\
        g_{\theta}(z)       & = g_{\theta}^{[L]}(z) = W^{[L-1]} g_{\theta}^{[L-1]}(z) + b^{[L-1]},
    \end{aligned}
\end{equation}
where each adaptive residual block is defined as
\begin{equation}
    \begin{aligned}
        & \text{AdaptiveResidualBlock}^{[l]}(g_{\theta}^{[l-1]}(z)) = \beta^{[l]} \cdot g_{\theta}^{[l-1]}(z) + (1 - \beta^{[l]}) \cdot h_{\theta}^{[l]}(z), \\
        & \qquad h_{\theta}^{[l]}(z) = \sigma \circ \left( W_2^{[l-1]} \sigma \circ ( W_1^{[l-1]} g_{\theta}^{[l-1]}(z) + b_1^{[l-1]} ) + b_2^{[l-1]} \right)
    \end{aligned}
\end{equation}
with $\beta^{[l]} \in [0.05, 1.0]$ being a learnable adaptive parameter for the $l$-th residual block. Each $\beta^{[l]}$ is initialized to $0.9$ and optimized during training. To ensure numerical stability, the value of each $\beta^{[l]}$ is constrained to the interval $[0.05, 1.0]$ via a clamping operation.
Here, $W_1^{[l]}, W_2^{[l]} \in \mathbb{R}^{m_{l+1} \times m_l}$ and $b_1^{[l]}, b_2^{[l]} \in \mathbb{R}^{m_{l+1}}$ are weight matrices and bias vectors, respectively. The input dimension is $m_0 = d_{\text{in}} = d$, and the output dimension is $m_L = d_0$. 
The activation function is chosen as $\sigma = \text{tanh}(\cdot)$, 
which is applied in the element-wise sense, as indicated by ``$\circ$''.
The collection of all network parameters, including $\{\beta^{[l]}\}_{l=1}^{L-1}$, is represented by $\theta$.

We fix the spatial domain to $[0, 1]^d$, and an exact periodic boundary condition is enforced to improve the numerical performance.
Taking the one dimensional case as an example, the ansatz is constructed based on a Fourier basis, where a transform $T: x \to \tilde{x} = \{\sin( 2\pi k x), \cos(2\pi k x)\}_{k = 1}^P$ is applied before the first layer of the deep neural network~\cite{han2020solving}. The extension to higher-dimensional cases follows naturally.
The averaging operator is evaluated numerically using a $16$-point Gauss–Legendre quadrature rule.

For network training, we employ the Adam optimizer with Xavier initialization. 
In each iteration, $4096$ random sample points are drawn from the spatial domain and $1024$ from the initial condition.
Except for the initial condition term, which is assigned a weight of $10$, all other loss terms are assigned a default weight of $1$.
To enhance training stability and convergence, we adopt an exponentially decaying learning rate schedule starting from
$\eta_0 = 10^{-3}$ with a decay rate of $\gamma = 0.96$ and a decay step of $p = 2500$ iterations, $\eta_{it} = \eta_0 \cdot \gamma^{\lfloor \frac{it}{p} \rfloor},$
where $it$ denotes the current iteration index, and the symbol $\lfloor \cdot \rfloor$ is the floor function.

The reference solutions are computed using a finite difference method~\cite{jin2000uniformly},
and we will check the relative $\ell^2$ error of the solution $s(x)$ of our method, e.g., for the one-dimensional case,
\begin{equation}
    \text{error} := \sqrt{
\sum_j |s_\theta^{j} - s_{\text{ref}}^j|^2 /
    {\sum_j |s_{\text{ref}}^j|^2}},
\end{equation}
where $s_\theta$ is the neural solution, and $s_{\text{ref}}$ is the reference solution.
Here, the superscript $j$ indexes the collocation points at which both $s_\theta$ and $s_{\text{ref}}$ are evaluated.

\subsection{One-dimensional example}
We first consider the one-dimensional RTE, testing the proposed framework across regimes ranging from the rarefied case ($\varepsilon = 1$) to the diffusive limit ($\varepsilon = 10^{-4}$). The computational domain is $x \in [0,1]$ with periodic boundary condition, and the initial condition is $f_{\rm IC}(x, v) = p(x)\exp(-v^2/2)/\sqrt{2\pi}$, where $p(x) = 1 + \cos(4\pi x).$

Fig.~\ref{fig: rte1d_history_1.0e+00} and Fig. \ref{fig: rte1d_rho_1.0e+00} show the training history and the evolution of the macroscopic density $\rho(t, x)$ for $\varepsilon = 1$ at $t = 0.05$ and $t = 0.1$. The loss and relative $\ell^2$ error decrease steadily as training proceeds, and the predicted densities closely match the reference finite difference solutions. 
This demonstrates that the proposed method effectively captures the kinetic behavior in the rarefied regime.
Similarly, the results for the diffusive regime ($\varepsilon = 10^{-4}$) are shown in Fig.~\ref{fig: rte1d_history_1.0e-04} and Fig. \ref{fig: rte1d_rho_1.0e-04}. As the regime approaches diffusion, the network continues to produce accurate solutions, maintaining numerical stability and precision even at smaller mean free path scales. The comparison between the neural and reference solutions again exhibits excellent agreement, verifying the robustness of the proposed method across multiscale regimes.

\begin{figure}[htbp!]
\centering
\subfigure[Training loss history]{
\includegraphics[width=0.4\textwidth]{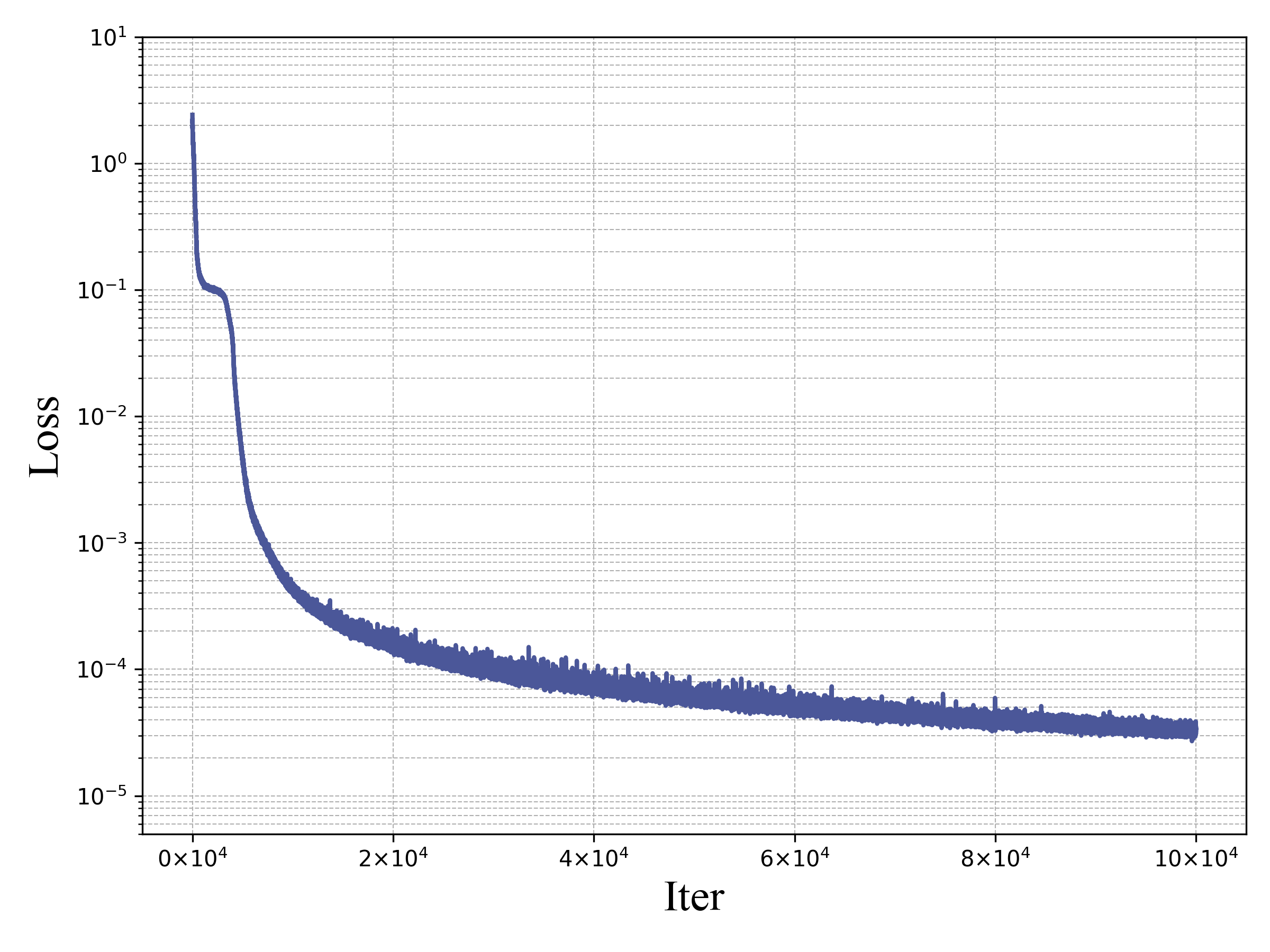}
}
\subfigure[Relative $\ell^2$ error]{
\includegraphics[width=0.4\textwidth]{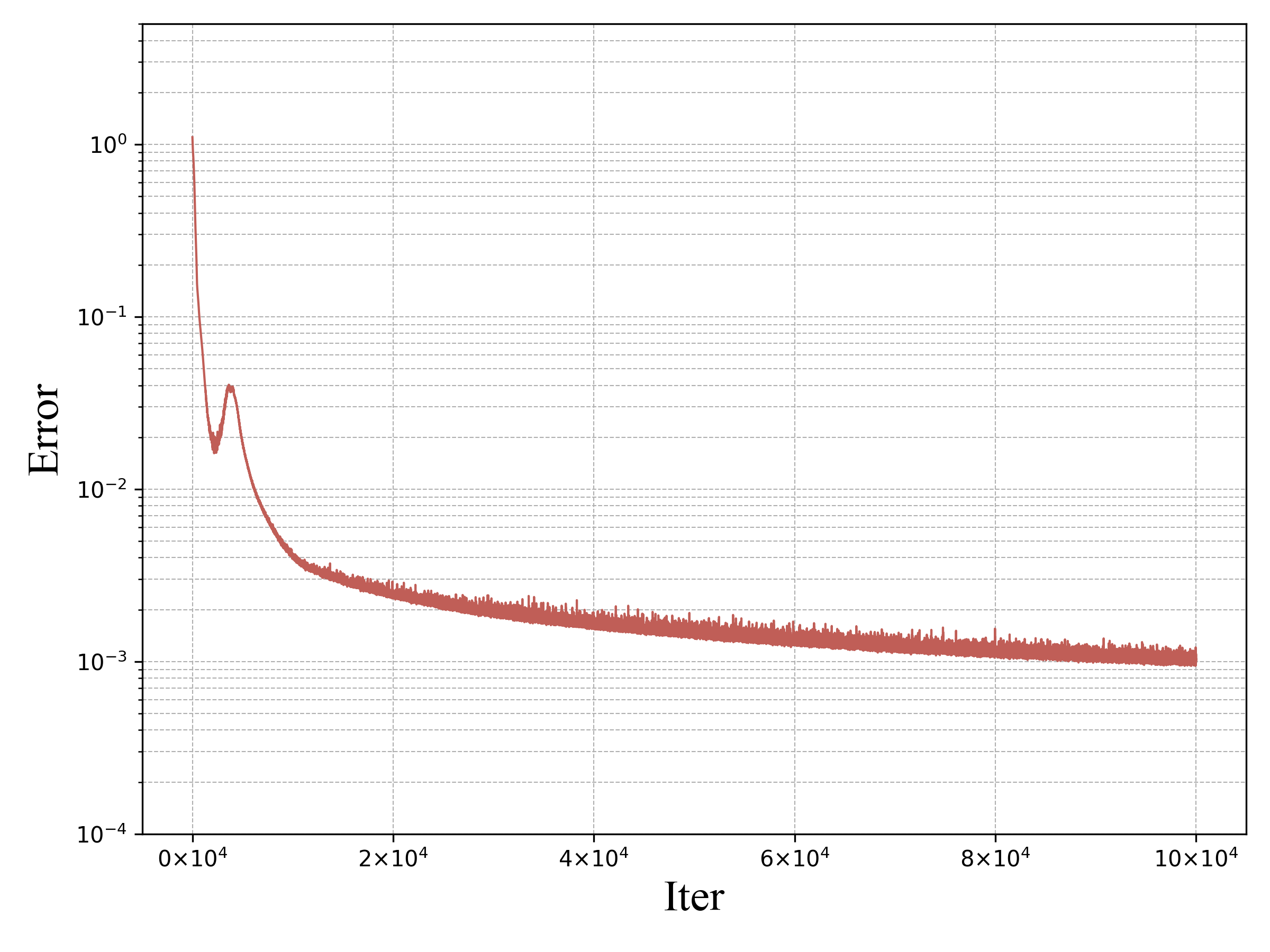}
}
\caption{Training loss and error evolution for the one-dimensional case ($\varepsilon = 1$). }
\label{fig: rte1d_history_1.0e+00}
\end{figure}

\begin{figure}[htbp!]
\centering
\includegraphics[width=0.5\linewidth]{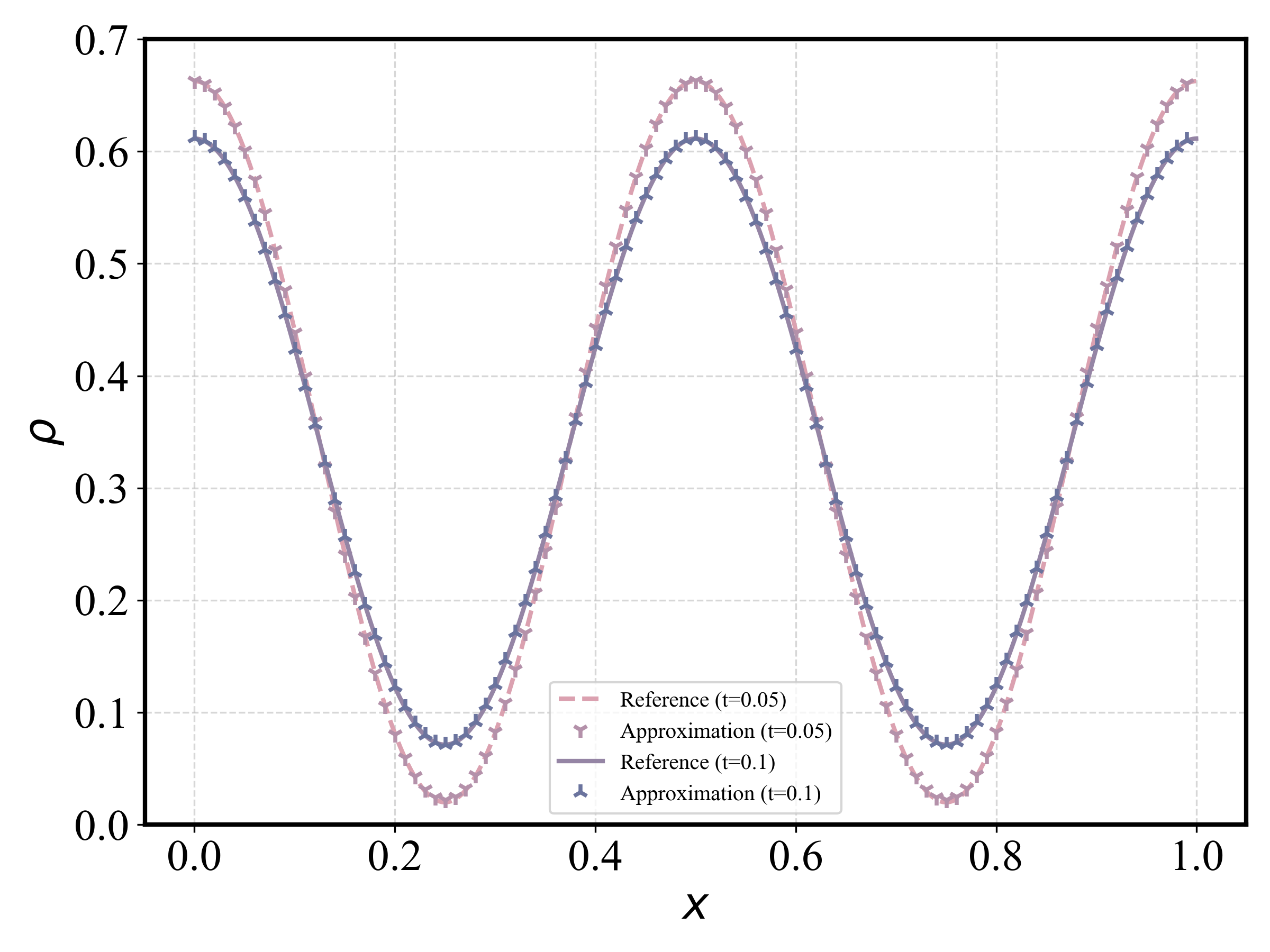}
\caption{Comparison between reference and neural network solutions at $t = 0.05$ and $t = 0.1$ ($\varepsilon = 1$).}
\label{fig: rte1d_rho_1.0e+00}
\end{figure}

\begin{figure}[htbp!]
\centering
\subfigure[Training loss history]{
\includegraphics[width=0.4\textwidth]{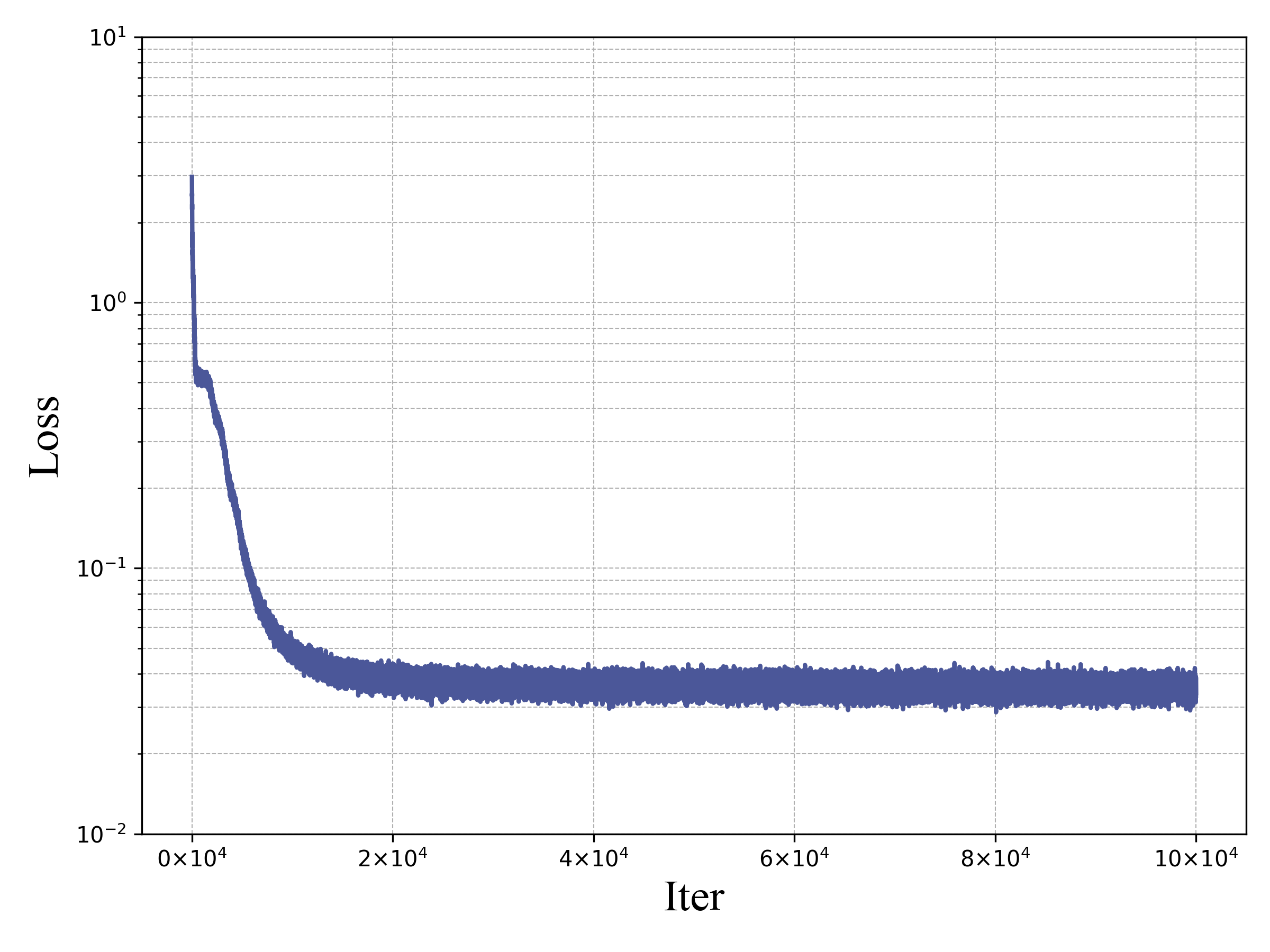}
}
\subfigure[Relative $\ell^2$ error]{
\includegraphics[width=0.4\textwidth]{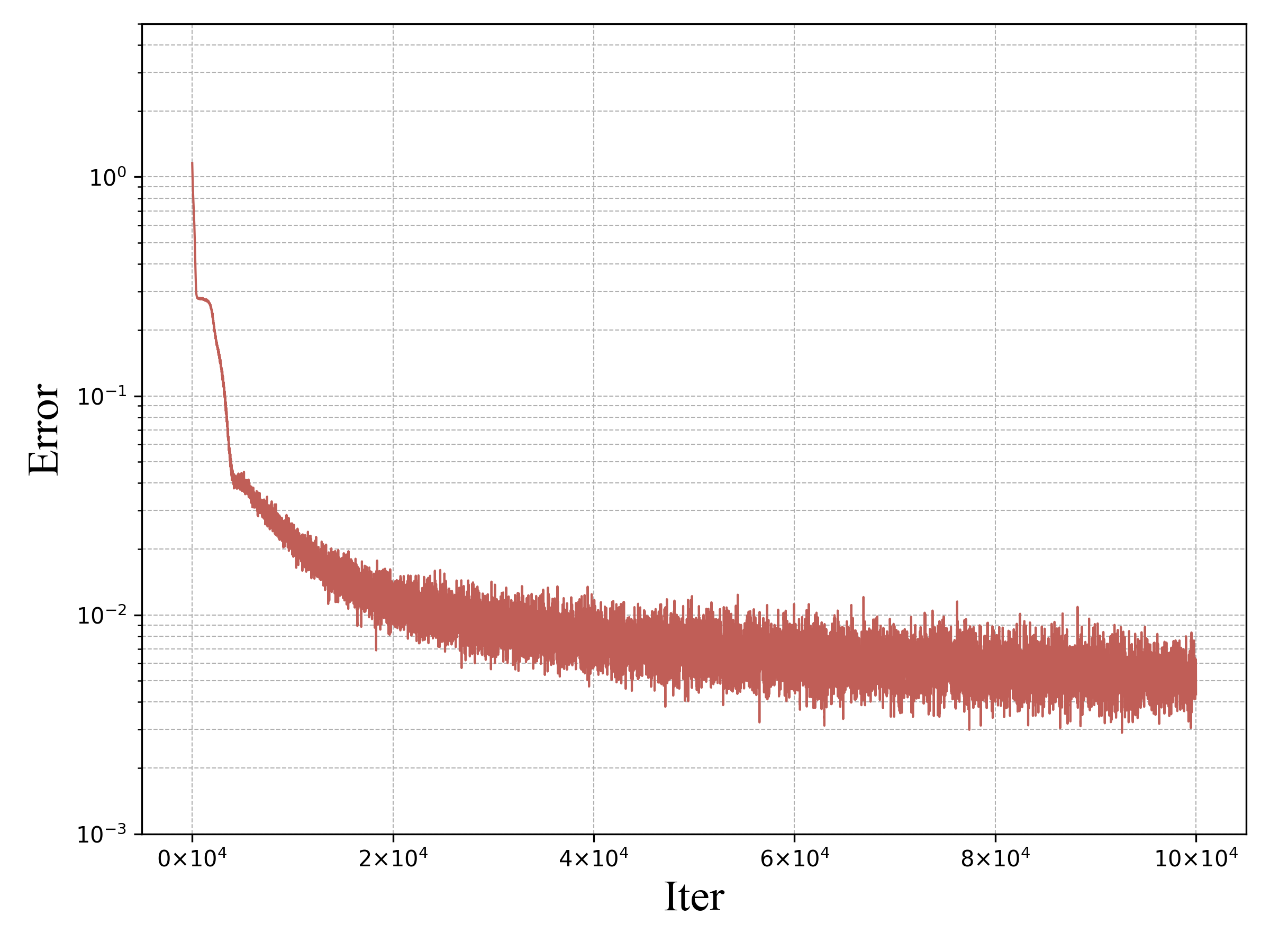}
}
\caption{Training loss and error evolution for the one-dimensional case ($\varepsilon = 10^{-4}$). }
\label{fig: rte1d_history_1.0e-04}
\end{figure}

\begin{figure}[htbp!]
\centering
\includegraphics[width=0.5\linewidth]{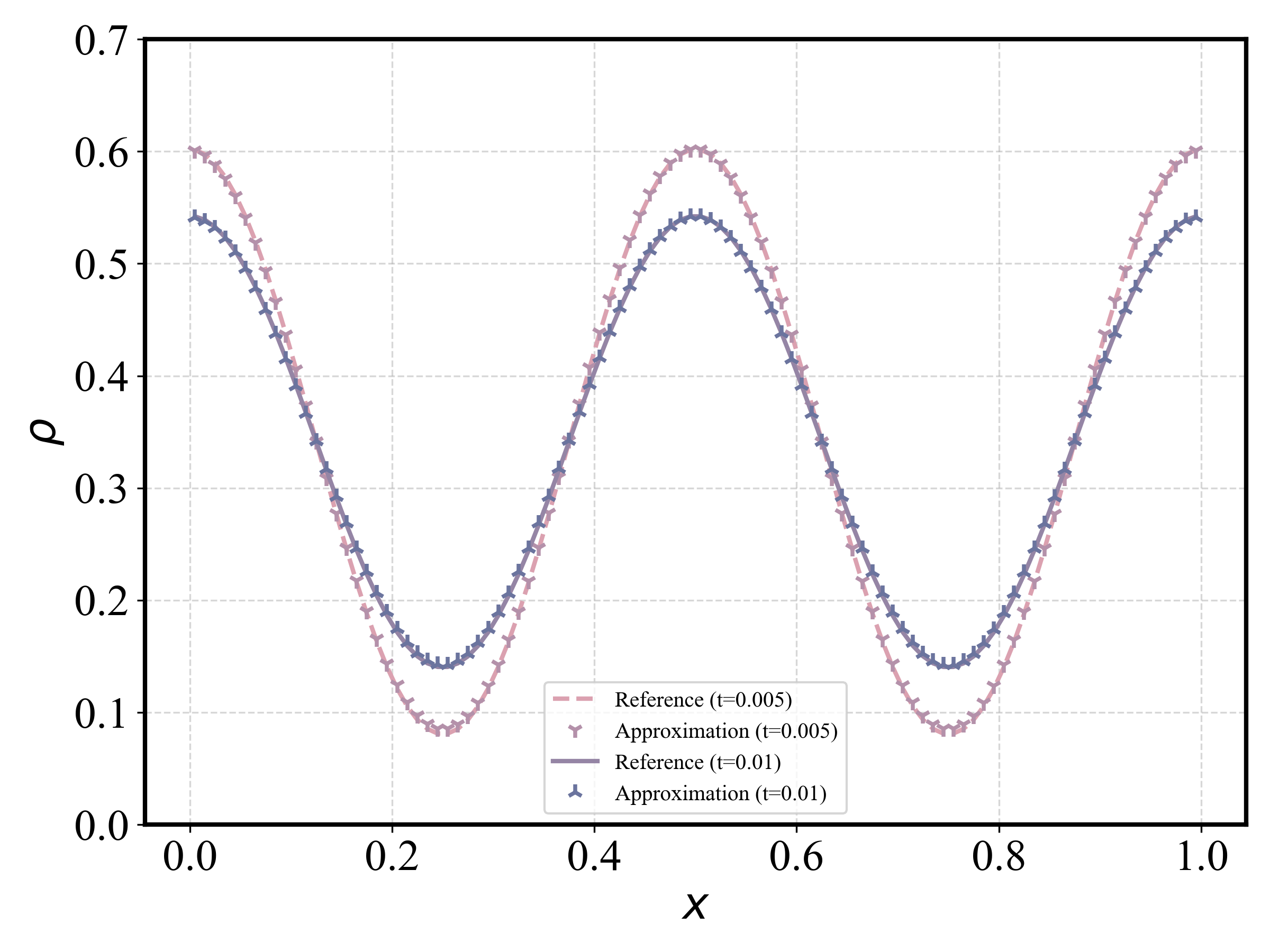}
\caption{Comparison between reference and neural network solutions at $t = 0.005$ and $t = 0.01$ ($\varepsilon = 10^{-4}$).}
\label{fig: rte1d_rho_1.0e-04}
\end{figure}

\subsection{Two-dimensional example}
We next extend the analysis to the two-dimensional RTE under periodic boundary conditions. The initial distribution is $f_{\rm IC}(\boldsymbol{x}, \boldsymbol{v})=p(\boldsymbol{x})\exp(-|\boldsymbol{v}|^2/2)/\sqrt{2\pi},$ where $p(\boldsymbol{x}) = 1 + (\cos(2\pi x_1) + \cos(2\pi x_2))/2.$

Fig.~\ref{fig: rte2d_rho_1.0e+00} and Fig. \ref{fig: rte2d_rho_1.0e-03} display the predicted and reference density fields $\rho(t, \boldsymbol{x})$ at time $t = 0.1$ for the rarefied regime ($\varepsilon = 1$) and the near-diffusive regime ($\varepsilon = 10^{-3}$), respectively. In both cases, our method accurately captures the spatial structure and amplitude of the density field, closely matching the reference solutions. These results confirm that the proposed adaptive framework generalizes effectively from one to two dimensions and maintains consistent performance across kinetic–diffusive transitions.

\begin{figure}[htbp!]
\centering
\includegraphics[width=0.75\linewidth]{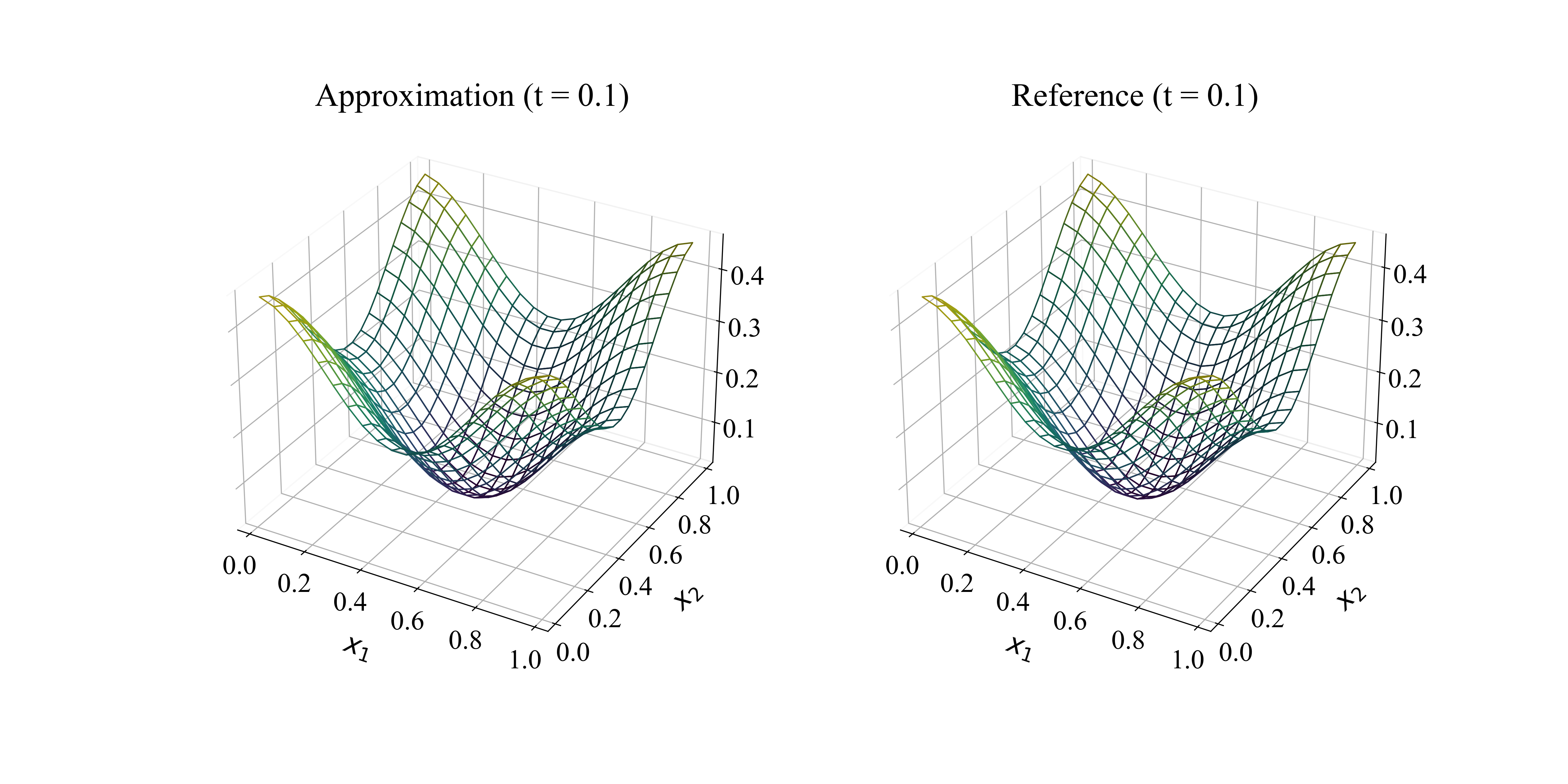}
\caption{Reference and approximate density fields at $t = 0.1$ ($\varepsilon = 1$).}
\label{fig: rte2d_rho_1.0e+00}
\end{figure}

\begin{figure}[htbp!]
\centering
\includegraphics[width=0.75\linewidth]{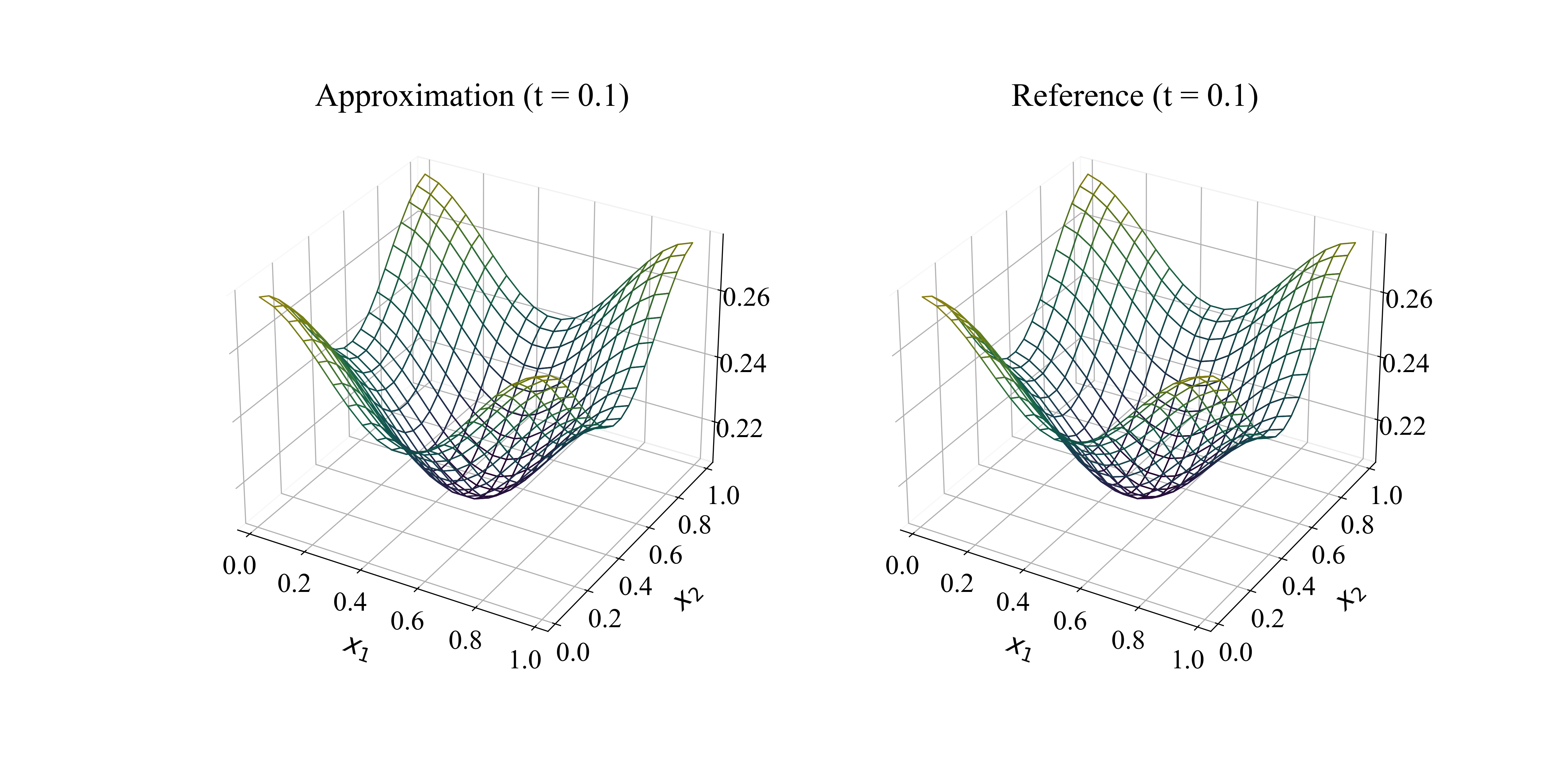}
\caption{Reference and approximate density fields at $t = 0.1$ ($\varepsilon = 10^{-3}$).}
\label{fig: rte2d_rho_1.0e-03}
\end{figure}

\section{Conclusions}\label{Sec: Conclusions}

In this paper, we have proposed a multiscale parity decomposition framework for solving radiative transfer equations. Both the theoretical analysis and numerical experiments demonstrate that the proposed framework achieves uniform error control while preserving the intrinsic structure of the system, thereby enabling stable and accurate computation for multiscale radiative transfer problems.
The main contributions of this work lie in the development of a structure-preserving and asymptotic-preserving neural network framework and the proof of uniform stability of the total error with respect to the Knudsen number, offering a viable approach for high-fidelity simulations of multiscale RTEs. The present study is limited to periodic boundary conditions; the extension to inflow and reflective boundary conditions will be addressed in future work. Moreover, the present work has focused on the one- and two-dimensional cases following \cite{jin2000uniformly}; extending the approach to three dimensions poses additional challenges due to increased geometric complexity, although the underlying theoretical framework is expected to remain consistent with the lower-dimensional settings. Furthermore, applying the proposed method to nonlinear kinetic systems remains an important direction for continued research.

\bibliographystyle{siamplain}
\bibliography{reference}

\end{document}